\documentclass[a4paper,11pt,leqno]{article}
\usepackage{relsize}
\usepackage{hyperref}
\hypersetup{colorlinks, linkcolor=blue, urlcolor=red, filecolor=green, citecolor=blue}
\usepackage{fancyhdr}
\medskipamount
\usepackage{comment}
\usepackage[active]{srcltx}

\usepackage{amsfonts}
\usepackage{amsmath}
\usepackage{amsthm}
\usepackage{amssymb}
\usepackage{amscd}
\usepackage{enumerate}
\usepackage{esint}
\usepackage{mathtools}
\usepackage{url}

\theoremstyle{plain}
\newtheorem{theorem}{Theorem}[section]

\newtheorem{lemma}[theorem]{Lemma}
\newtheorem{proposition}[theorem]{Proposition}

\theoremstyle{definition}
\newtheorem{assumption}[theorem]{Assumption}
\newtheorem{definition}[theorem]{Definition}

\theoremstyle{remark}
\newtheorem{remark}{Remark}
\newtheorem*{remarks}{Remarks}

\newcommand\R{\mathbb R}

\newcommand\Z{\mathbb Z}
\newcommand\M{\mathbb{M}}
\newcommand\ud{\,\mathrm{d}}

\newcommand{\e}{\varepsilon}
\newcommand{\eh}{\varepsilon(h)}

\newcommand{\SO}[1]{\operatorname{SO}(#1)}
\newcommand{\so}[1]{\operatorname{so}(#1)}
\newcommand\id{I}
\newcommand\sym{\operatorname{sym}}

\newcommand\skw{\operatorname{skw}}
\newcommand{\WW}{{H^{2}_{\delta}}}
\newcommand\tr{\operatorname{tr}}
\newcommand\dist{\operatorname{dist}}

\newcommand\wto{\rightharpoonup}

\newcommand\wtto{\xrightharpoonup{2,0}}
\newcommand\otto{\xrightharpoonup{osc,0}}
\newcommand\stto{\xrightarrow{2,0}}
\renewcommand{\d}{\partial}
\DeclareMathOperator{\spt}{spt}
\DeclareMathOperator{\secf}{II}

\DeclareMathOperator{\cof}{cof}

\renewcommand{\t}{\widetilde}

\newcommand\esssup{\mathop{\operatorname{ess\,sup}}}

\title{On the derivation of homogenized bending plate model}
\date{\today}

\author{Igor Vel\v{c}i\'c,\\
University of Zagreb, Faculty of Electrical Engineering and Computing\\
  Unska 3, 10000 Zagreb, Croatia,\\
  igor.velcic@fer.hr}

\begin{document}

\maketitle

\begin{abstract}
We derive, via simultaneous homogenization and dimension reduction,
the $\Gamma$-limit for thin elastic plates of thickness $h$ whose energy density
oscillates on a scale $\eh$ such that $ \eh^2 \ll h\ll \eh$. We consider the energy scaling that corresponds
to Kirchhoff's nonlinear bending theory of plates.

 \vspace{10pt}
\end{abstract}



\section{Introduction}

In this paper we derive a model for homogenized bending plate, by means of $\Gamma$-convergence, from $3D$ nonlinear elasticity.
The pioneering papers in derivation of lower dimensional models by means of $\Gamma$ convergence are \cite{AcBuPe} where the equations for elastic string are derived and \cite{LDRa95} where membrane plate equations are derived. It was well known that the obtained models depend on the assumption on the order of external loads with respect to the thickness of the body (see \cite{Ciarletplates} for the approach via formal asymptotic expansion). The pioneering works in deriving higher dimensional models (e.g. bending and von-K\'arm\'an) via $\Gamma$-convergence are \cite{FJM-02,FJM-06} where the key indegredient is the theorem on geometric rigidity.

Recently, models of homogenized bending plate were derived by simultaneous homogenization and dimensional reduction in the special case when the relation between the thickness of the body $h$  and the oscillations of the material $\eh$ satisfy the condition $h\thicksim \eh$ or $\eh \ll h$ i.e. the situations when is such that $\lim_{h \to 0} \tfrac{h}{\eh}=:\gamma \in (0,\infty]$, see \cite{Horneuvel12}. Here we partially cover the case $\gamma=0$, by assuming additionally $\eh^2 \ll h \ll \eh$.
We emphasize the governing approach: since we are in small strain regime, the  energy is essentially convex in strain. Since the strain  is a nonlinear function of the gradient  it is not a-priori guaranteed  that the two scale analysis will be enough to obtain $\Gamma$-limit since this happens only for the energies that are convex in gradient.
 However in this regime, as well as in the regimes studied in \cite{Horneuvel12}, it is enough to include only the oscillations that follow the oscillations of the material to obtain relaxation formula (as a consequence of convexity). As explained below for the case $h\sim \varepsilon(h)^2$ one should not expect this.

Let us emphasize the fact that the result of this paper and \cite{Horneuvel12}  show that, in order to obtain the right answer on the question how to homogenize the bending model of plates, one needs to look full $3D$ equations and perform simultaneous homogenization and dimensional reduction. The result of this paper specially shows that, even in the case when dimensional reduction dominates,  for some regimes $2D$ equations of bending plate do not give enough  information how to do the homogenization (this does not happen in the case of von K\'arm\'an plate).
Namely, the homogenization under the restriction of being an isometry would cause more rigid behavior than the one presented here as suggested by the work of S. Neukamm and H. Olbermann (see \cite{NeuOlb-13}).

The homogenization of von-K\'arm\'an model of plate and shells are discussed in \cite{NeuVel-12,Hornungvel12}.
In the case of von-K\'arm\'an plate all cases for $\gamma$ are obtained; the case $\gamma=0$ corresponds to the situation where dimensional reduction prevails and the case $\gamma=\infty$ corresponds to the situation where homogenization prevails. Both of these cases can be obtained as limit cases from the intermediate thin films that arises when $\gamma \to 0$, i.e., $\gamma \to \infty$.
The case $\gamma=0$ corresponds to the homogenized model of $2D$ von-K\'arm\'an plate; the case $\gamma=\infty$ corresponds to the case of dimensional reduction of homogenized material. The similar phenomenology occurs in the case of bending rod (see \cite{Neukamm-10,Neukamm-11}).
In \cite{Neukamm-10} the author also opened the question on homogenization of bending plate.

For the derivation of von K\'arm\'an shell model from $3D$ elasticity, without homogenization, see \cite{Lewicka1}.
In the case of homogenization of von K\'arm\'an shell the surprising fact  was there are also different scenarios for $\gamma=0$; we were able to identify different models for $\eh^2 \ll h \ll \eh$ and for $h\thicksim \eh^2$. Recently simultaneous homogenization and dimensional reduction without periodicity assumption was performed to obtain the equations of (homogenized) von K\'arm\'an plate (see \cite{Velcic-13}). The arguments presented here strongly suggest that such general approach is not possible in the case of bending plate (as the arguments in \cite{Hornungvel12} suggest that such approach is not possible in the case of von K\'arm\'an shell), due to more complex phenomenology.

The recovery sequence for this model is significantly different then the one defined in \cite{Horneuvel12} and its gradient includes the terms of order $\eh\gg  h$, which then has to be of the specific form, in order to obtain the energy of order $h^2$ (see the expression (\ref{fin-1}) below).
The basic heuristic for the existence of this model lies in that fact. Namely, in the limiting procedure it is possible to add the oscillations in the normal direction  that will cause that the recovery sequence $(u^h)_{h>0}$ satisfies
$$\nabla_h u^h=R+\varepsilon(h)RA^h+O(h).$$
 $R$ is a smooth matrix field taking the values in $\SO 3$ and $A^h$ is a smooth matrix field taking the values in the space of skew symmetric matrices.
This causes the fact that the strain has the order $h$ (using the assumption $\varepsilon(h)^2 \ll h \ll \varepsilon(h)$), since
$$ (\nabla_h u^h)^t \nabla_h u^h=I+O(h)+\varepsilon(h)^2 (A^h)^t A^h.$$
As a consequence we have that the energy is of order $h^2$ , i.e., we are in the bending regime.
The difficulty arises in proving that these are all the relevant effects for relaxation, i.e., in proving what is the two scale limit of the strain. From this simple heuristic it can also be seen why  the regime $h \sim \varepsilon(h)^2$ is critical. Namely, in that case the term $(A^h)^t A^h$ would become relevant in the leading order of the strain, which can cause that the oscillations of order different than $\varepsilon(h)$ could become energetically more optimal (see Remark \ref{slucaj2} below).
This nonlinear term is also reflecting the lack of compactness result, see Remark \ref{slucaj1} below.
However, in the case $h \sim \varepsilon(h)^2$ we are not able to identify the model.
Also this simple heuristic explains why would the homogenization of the bending plate functional under the restriction of being an isometry cause more rigid behavior. Namely, these oscillations in the normal direction for an arbitrary point of the domain would then be prevented by the constraint of being an isometry.
Let us also emphasize the fact that this model can be obtained from the models given in \cite{Horneuvel12} by letting $\gamma \to 0$ (see Remark \ref{rem:poveznica} below). The same phenomenology occurred in the case of von K\'arm\'an shell for the situation $\varepsilon(h)^2 \ll h \ll \varepsilon(h)$, when the continuity in $\gamma$ was preserved.

The compactness result,  given in \cite{FJM-02}, which gives the lower bound of $\Gamma$-limit,  forces us to work with piecewise constant map with values in $\SO 3$ which creates some additional technical difficulties in the compactness lemmas that are needed for recognizing the oscillatory part of two scale limit (see Lemma \ref{lema:glavna1} below). The proof of the compactness result (i.e., identification of two scale limit of strain) is different from the one in \cite{Horneuvel12}. The claims that lead to that proof as well as realizing the possibility of oscillations in the normal direction whose gradient is of order $\varepsilon(h)\gg h$ are the main contributions of this paper with respect to \cite{Horneuvel12}.
The mathematical contribution of series of papers dealing with simultaneous homogenization and dimensional reduction in elasticity (see \cite{Neukamm-10,Neukamm-11,Hornungvel12,Horneuvel12}) consists in adapting two scale convergence methods to the results given in \cite{FJM-02,FJM-06} to obtain more information on the limiting strain. This is, however, non trivial task due to the fact that the compactness results in \cite{FJM-02,FJM-06} are subtle and every situation requires its own adaption.

By $\so n$ we denote the space of skew symmetric matrices in $\M^n$, while by $\M^n_{\sym}$ the space of symmetric matrices in $\M^n$. For a matrix $A$, $\sym A$ denotes its symmetric part while $\skw A$ denotes its skew symmetric part i.e. $\sym A=\tfrac{1}{2}(A+A^t)$,
$\skw A=\tfrac{1}{2}(A-A^t)$. By $\tr A$ we denote the trace of the matrix $A$, while by $|A|$ we denote the Frobenius norm of $A$, $|A|=\tr(A^t A)$.
For a vector $v \in \R^3$ by
$A_{v}$ we denote the skew symmetric tensor given by
$A_{v} x=v \times x$.
We call $v$ the axial vector of $A_v$.
One easily obtains
\begin{equation}
A_{v}=\left(\begin{array}{ccc} 0 &-v_3  & v_2 \\ v_3 & 0 & -v_1\\-v_2 & v_1 &0  \end{array} \right).
\end{equation}
For $v, w \in \R^3$ we have
\begin{equation}
A_{v} \cdot A_{w}=2  v \cdot  w.
\end{equation}
By $\delta_{ij}$ we denote the Kronecker delta. $A\lesssim B$ means that the inequality is valid up to a multiplicative constant $C>0$ on the right hand side.

We need to introduce some function spaces of periodic functions.
From now on, $Y = [0,1)^2$, and we denote by $\mathcal Y$ the set $Y$ endowed
with the topology and differential structure of $2$-torus, so that functions on $\mathcal Y$ will be $Y$-periodic.
We will identify the functions on $\mathcal{Y}$ with their periodic extension on $\R^2$.
\\
We write $C(\mathcal Y)$, $C^k(\mathcal Y)$ and $C^\infty(\mathcal Y)$ for the
spaces of $Y$-periodic functions on $\R^2$ that are continuous,
$k$-times continuously differentiable and smooth, respectively.
Moreover $H^k(\mathcal Y)$ denotes the closure of
$C^\infty(\mathcal Y)$ with respect to the norm in
$H^k(Y)$.  This has the consequence that the space $H^k(\mathcal Y)$ consists of functions that are in $H^k(Y)$ and have the derivatives up to order $k-1$ that are $Y$-periodic (taking into account the traces). The space $L^2(\mathcal Y)$ is equal to the space $L^2(Y)$.

In the same way, $H^k(I\times\mathcal Y)$ (here $I=[-\tfrac{1}{2},\tfrac{1}{2}]$) denotes the closure of
$C^\infty(I,C^\infty(\mathcal Y))$ with respect to the norm in
$H^k(I\times Y)$.

We write $\mathring H^k(\mathcal Y)$ for the subspace of
functions $f\in H^k(\mathcal Y)$ with $\int_{Y}f=0$.
In the analogous way we define $\mathring H^1(I \times\mathcal Y)$, $\mathring{C}^k(\mathcal{Y})$.
The definitions extend in the obvious way to vector-valued functions.

\section{General framework and main result}\label{S:framework}

From now on, $S\subset\R^2$ denotes a bounded Lipschitz domain whose boundary is piecewise $C^1$.
The piecewise $C^1$-condition is necessary only for the proof of the upper bound and can be slightly
relaxed, cf. Section \ref{Upper Bound} for the details.
By $\tilde{S}$ we denote some subset of $S$ with Lipschitz boundary, unless stated otherwise.
\\
For $h > 0$ and $I:=[-\frac{1}{2},\frac{1}{2}]$, we denote by
$\Omega_h:=S\times hI$ the reference configuration of the thin plate of thickness $h$.
By $\tilde{\Omega}_h$ we denote $\tilde{\Omega}_h:=\tilde{S}\times hI$, where $\tilde{S}$ is defined above.
The elastic energy per unit volume associated with a deformation $v^h:\Omega_h\to\R^3$
is given by
\begin{equation}\label{Neu:eq:00}
  \frac{1}{h}\int_{\Omega_h}W( \frac{z_3}{h},\frac{z'}{\e},\nabla v^h(z))\,dz.
\end{equation}
Here and below $z'=(x_1,x_2)$ stands for the in-plane coordinates of a generic
element $z=(x_1,x_2,x_3)\in\Omega_h$ and $W$ is a energy
density that models the elastic properties of a periodic composite.
\begin{assumption}\label{assumption}
  We assume that
  \begin{equation*}
    W: [-\tfrac{1}{2},\tfrac{1}{2}] \times \R^2\times\M^3 \to [0,\infty],\qquad (x_3,y,F)\mapsto W( x_3,y,F),
  \end{equation*}
  is measurable and $[0,1)^2$-periodic in $y$ for all $F$. Furthermore, we assume that for
  almost every $( x_3,y)\in [-\tfrac{1}{2},\tfrac{1}{2}] \times \R^2$, the map $\M^3\ni F\mapsto W( x_3, y,F)\in
  [0,\infty]$ is continuous and satisfies the following properties:
  \begin{align}
    \tag{FI}\label{ass:frame-indifference}
    &\text{(frame indifference)}\\
    &\notag\qquad W(x_3,y, RF)=W(x_3,y, F)\quad\text{ for
      all $F\in\M^3$, $R\in\SO 3$;}\\[0.2cm]
    \tag{ND}\label{ass:non-degenerate}
    &\text{(non degeneracy)}\\
    &\notag\qquad W(x_3,y,  F)\geq c_1\dist^2( F,\SO 3)\quad\text{ for all
      $ F\in\M^3$;}\\
    &\notag\qquad W(x_3, y, F)\leq c_2\dist^2( F,\SO 3)\quad\text{ for all
      $ F\in\M^3$ with $\dist^2( F,\SO 3)\leq\rho$;}\\[0.2cm]
    \tag{QE}\label{ass:expansion}
    &\text{(quadratic expansion at identity)}\\
    &\notag\qquad
    \lim_{G\to 0}\esssup_{y \in \mathcal{Y}}\frac{|W(x_3,y, \id+ G) - Q(x_3,y, G)|}{|G|^2} = 0,\\
    &\notag\qquad\text{for some quadratic form $Q(x_3,y,\cdot)$ on $\M^3$.}
  \end{align}
  Here $c_1,c_2$ and $\rho$ are positive constants which are
  fixed from now on.
\end{assumption}
As a direct consequence of (QE) we have that
 there exists a monotone function $r:\R_+\to\R_+\cup\{+\infty\}$,
  such that $r(\delta)\to 0$ as
  $\delta\to 0$ and
  \begin{equation}\label{eq:94}
    \forall  G\in\M^{3}\,:\,|W( x_3,y,\id+ G)-Q(x_3, y, G)|\leq| G|^2r(| G|).
  \end{equation}
  for  almost every $y\in\R^2$.

We define $\Omega := S\times I$. As in \cite{FJM-02} we rescale the out-of-plane coordinate: for
$x = (x',x_3)\in\Omega$ consider the scaled deformation
$u^h(x',x_3):=v^h(x',hx_3)$. Then \eqref{Neu:eq:00} equals
\begin{equation}
  \label{eq:3'}
  \mathcal E^{h,\e}(u^h):=\int_{\Omega}W(x_3,\frac{x'}{\e},\nabla_h u^h(x))\,dx,
\end{equation}
where $\nabla_h u^h:=\big(\,\nabla' u^h,\,\frac{1}{h}\partial_3
u^h\,\big)$ denotes the scaled gradient, and
$\nabla'u^h:=\big(\,\partial_1u^h,\,\partial_2 u^h\,\big)$ denotes the
gradient in the plane.

As already mentioned in the introduction, it is well known that different models for thin bodies can be obtained from three dimensional elasticity equations by the method of $\Gamma$-convergence. The main assumption that influences the derivation of the model is the assumption on the relation of the order of the energy and the thickness of the body (also the assumption on the boundary conditions can influence the model). The  plate behavior, due to its more complex geometry, is much more complex than the behavior of rods.
We recall some known results on dimension reduction in the homogeneous case when  $W(x_3,y,F)=W(F)$.
In \cite{FJM-06} a hierarchy of
plate models is derived from $\mathcal E^h:=\mathcal E^{h,1}$ in the zero-thickness limit $h\to 0$.
The case  $\mathcal E^h\sim 1$ leads to a membrane model (see \cite{LDRa95}), which is a fully nonlinear plate
model for plates without resistance to compression.
The reason for that is that the compression enables the plate to preserve the metric on the mid-plane and thus these deformations have lower order energy.
The obtained equations are of the same type as the original $3D$ equations i.e. quasilinear and of the second order.

In this article we study the \textit{bending regime} $\mathcal
E^h\sim h^2$, which, as shown in \cite{FJM-02}, leads to Kirchhoff's nonlinear plate
model: as $h\to 0$ the energy $h^{-2}\mathcal E^h$
$\Gamma$-converges to the functional
\begin{equation}\label{Neu:eq:1}
  \int_S Q_2(\secf(x'))\,dx',
\end{equation}
with $Q_2 : \M^2\to\R$
given by the relaxation formula
\begin{equation}\label{dimredukcija}
  Q_2(A)=\min_{d\in\R^3}Q\left(\iota(A)+d\otimes e_3\right);
\end{equation}
here, $Q$ denotes the quadratic form from \eqref{ass:expansion} and $\iota$ denotes the natural injection of $\M^2$ into $\M^3$.

Denoting the standard basis of $\R^3$ by $(e_1,e_2,e_3)$ it is given by
$$ \iota(A)= \sum_{\alpha,\beta=1}^2 A_{\alpha \beta} (e_\alpha \otimes e_\beta).$$

The special case of layered material is considered in \cite{Schmidt-07}. Dependence on $x_3$ variable produces non-trivial effects on the relaxation formula. Namely the limit functional is then given by
\begin{equation}\label{Neu:eq:111}
  \int_S \bar{Q}_2(\secf(x'))\,dx',
\end{equation}
where
\begin{equation}\label{dimredukcija2}
  \bar{Q}_2(A)=\min_{B\in\M^2_{\sym} }\int_I Q_2\left(x_3,x_3A+B\right)\, dx_3.
\end{equation}

In \cite{Horneuvel12} it was shown that in the non-homogeneous case the effective
quadratic form $Q_2$ is determined by a relaxation formula that
is more complicated. For construction of the recovery sequence it was also helpful to understand the behavior of layered materials.
The obtained model
depended on the relative scaling between the thickness $h$ and the
material period $\e$.
 Here we will make the assumption that $\e$ and $h$
are coupled as follows:

\begin{assumption}\label{assumption:gamma}
 We assume that $\e=\eh$ is a nondecreasing
 function from $(0,\infty)$ to $(0,\infty)$ such that $\eh\to 0$ and
 $\eh^2 \ll h \ll \eh$ i.e. $\lim_{h \to 0} \tfrac{h}{\eh}=\lim_{h \to 0} \tfrac{\eh^2}{h}=0$.
\end{assumption}

The energy density of the homogenized plate we derive here is given by means of a relaxation formula that we
introduce next.

\begin{definition}[Relaxation formula]
  \label{D:relaxation-formula}
  Let $Q$ be as in Assumption~\ref{assumption}.
  We  define
    $Q_{2}^{\text{rel}} :\M^2 _{\sym}\to[0,\infty)$ by
\begin{eqnarray*}
Q_{2}^{\text{rel}}(A)&:=&
\inf_{B,U}\iint_{I\times Y}Q\left(x_3,y,\;\iota(x_3A+B)+U\,\right)\,dy\,dx_3,
\end{eqnarray*}
    taking the infimum over all $B\in\M^{2}_{\sym}$
    and $U \in  L_0(I\times\mathcal{Y},\M^3_{\sym}) $,
where
 \begin{eqnarray*}
    L_0(I\times\mathcal Y,\M^3_{\sym})&:=&\Bigg\{\,
    \left(\begin{array}{cc}
        \sym\nabla_y\zeta+x_3\nabla_y^2 \varphi
        & \begin{array}{cc}  g_1\\  g_2\end{array}\\
        ( g_1,\; g_2)& g_3\\
      \end{array}\right)\,:\;\zeta\in \mathring{H}^1(\mathcal Y,\R^2),\\
    &&\qquad\varphi\in \mathring{H}^2(\mathcal
    Y),\,\ g\in L^2(I\times\mathcal Y,\R^3)\, \Bigg\}.
\end{eqnarray*}
We also define the mapping $\mathcal{U}:\mathring{H}^1(\mathcal Y,\R^2) \times  \mathring{H}^2(\mathcal Y) \times L^2(I \times\mathcal Y,\R^3) \to L_0(I\times\mathcal Y,\M^3_{\sym})$ by
$$ \mathcal{U}(\zeta,\varphi,g)=\left(\begin{array}{cc}
        \sym\nabla_y\zeta+x_3\nabla_y^2 \varphi
        & \begin{array}{cc}  g_1\\  g_2\end{array}\\
        ( g_1,\; g_2)& g_3\\
      \end{array}\right). $$
\end{definition}
For a simpler definition of $Q_2^{\text{rel}}$, see Remark \ref{rem:jednostavno}.
\begin{remark} \label{kon20}
It can be easily seen by using Korn's inequality that for $\zeta \in \mathring{H}^1(\mathcal Y,\R^2)$,
$\varphi \in  \mathring{H}^2(\mathcal Y)$, $g \in L^2(I \times\mathcal Y,\R^3)$ we have
\begin{equation}\label{eq:hovel1}
\|\zeta\|_{H^1}+\|\varphi\|_{H^2}+\|g\|_{L^2} \lesssim \| \mathcal{U}(\zeta,\varphi,g)\|_{L^2}.
\end{equation}
\end{remark}

We define the constraint on the deformation to be an isometry:
\begin{equation}
  \label{eq:isometry}
  \partial_\alpha
u\cdot\partial_\beta u=\delta_{\alpha\beta},\qquad\alpha,\beta\in\{1,2\}.
\end{equation}
We also define the set of isometries of $S$ into $\R^3$
\begin{eqnarray}
  \WW(S,\R^3) &:=& \Bigg\{\,u\in H^2(S,\R^3)\, : \, u\text{ satisfies }\eqref{eq:isometry}\mbox{ a.e. in }S\,\Bigg\}.
\end{eqnarray}
With each $u \in \WW(S)$ we associate
its normal $n :=\partial_1 u \wedge \partial_2 u$,
and we define its second fundamental form $\secf: S\to\M^2_{\sym}$
by defining its entries as
\begin{equation}
  \secf_{\alpha\beta} = \partial_\alpha u\cdot\partial_\beta n
= -\partial_{\alpha}\partial_{\beta} u\cdot n.
\end{equation}
We write $\secf^h$ and $n^h$ for the second fundamental form
and normal associated with some $u^h\in\WW(S,\R^3)$.
The $\Gamma$-limit is a functional of the form \eqref{Neu:eq:1}
trivially extended to $L^2(\Omega,\R^3)$ by infinity: we define $\mathcal E:L^2(\Omega,\R^3)\to[0,\infty]$,
\begin{equation*}
  \mathcal E(u):=
  \left\{%
    \begin{aligned}
      &\int_S Q_{2}^{\text{rel}}(\secf(x'))\,dx'&\quad&\text{if }u\in\WW(S,\R^3),\\
      &+\infty&&\text{otherwise.}
    \end{aligned}%
    \right.
\end{equation*}
We identify functions on $S$ with their trivial
extension to $\Omega=S\times I$: $u(x',x_3):= u(x').$

Our main result is the following:
\begin{theorem}
  \label{T:main}
  Suppose that Assumptions~\ref{assumption} and \ref{assumption:gamma}
  are satisfied. Then:
  \begin{enumerate}[(i)]
  \item (Lower bound). If $(u^h)_{h>0}$ is a sequence with
    $u^h-\fint_\Omega u^h\,dx\to
    u$ in $L^2(\Omega,\R^3)$, then
    \begin{equation*}
      \liminf\limits_{h\to 0}h^{-2}\mathcal E^{h,\eh}(u^h)\geq \mathcal E(u).
    \end{equation*}
  \item (Upper bound). For every $u\in\WW(S,\R^3)$ there exists a
    sequence $(u^h)_{h>0}$ with $u^h\to u$ strongly in
    $H^1(\Omega,\R^3)$ such that
    \begin{equation*}
      \lim\limits_{h\to 0}h^{-2}\mathcal E^{h,\eh}(u^h)=\mathcal E(u).
    \end{equation*}
  \end{enumerate}
\end{theorem}

The first step for identifying the $\Gamma$-limit is the compactness result. It gives us the information on the limit deformations that can be concluded from the smallness of energy. The result is given in \cite{FJM-02} and relies on the theorem on geometric rigidity which is the key mathematical indegredient for deriving lower dimensional models.

\begin{theorem}[\mbox{\cite[Theorem~4.1]{FJM-02}}]
  \label{T:FJM-compactness}
  Suppose a sequence $u^h\in H^1(\Omega,\R^3)$ has finite
  bending energy, that is
  \begin{equation*}
    \limsup\limits_{h\to 0}\frac{1}{h^2}\int_{\Omega}\dist^2(\nabla_h
    u^h(x),\SO 3)\,dx<\infty.
  \end{equation*}
  Then there exists $u\in
  H^2_{\delta}(S,\R^3)$ such that
  \begin{align*}
    u^h-\fint_\Omega u^h\,dx&\to u,&\qquad&\text{strongly in }L^2(\Omega,\R^3),\\
    \nabla_h u^h&\to(\,\nabla'u,\,n\,)&\qquad&\text{strongly in
    }L^2(\Omega,\R^{3\times 3}),
  \end{align*}
  as $h\to 0$ after passing to subsequences and extending $u$ and $n$ trivially to $\Omega$.
\end{theorem}
\begin{remark}
We could also look the energy density dependent additionally on the macroscopic variable $x' \in \omega$ i.e. look the energy functionals of the form
\begin{equation}\label{Neu:eq:0}
  \frac{1}{h}\int_{\Omega_h}W(z', \frac{z_3}{h},\frac{z'}{\e},\nabla v^h(z))\,dz.
\end{equation}
This changes the relaxation formula in an obvious way, but creates some additional technical considerations, see \cite{NeuVel-12}.
\end{remark}

\section{two scale limits of the nonlinear strain}\label{S:two scale}

Two scale convergence was introduced in \cite{Nguetseng-89,
  Allaire-92} and has been extensively applied to various problems in
homogenization.
 This is mainly related to convex energies for which it is known that the oscillation that relax the energy are of the same type as the oscillations of the material. In the non convex case more complex behavior is expected.
In this article we work with the following variant of
two scale convergence which is adapted to dimensional reduction.
\begin{definition}[two scale convergence]
  \label{D:two scale-thin}
  We say a bounded sequence $\{f^h\}_{h>0}$ in $L^2(\Omega)$ two scale
  converges to $f\in L^2(\Omega\times Y)$ and we write $f^h\wtto f$, if
  \begin{equation*}
    \lim\limits_{h\to 0}\int_\Omega
    f^h(x)\psi(x,\frac{x'}{\eh})\,dx=\iint_{\Omega\times Y} f(x,y)\psi(x,y)\,dy\,dx,
  \end{equation*}
  for all $\psi\in C^\infty_0(\Omega,C(\mathcal Y))$. When
  $||f^h||_{L^2(\Omega)}\to||f||_{L^2(\Omega\times Y)}$ in addition,
  we say that $f^h$ strongly two scale converges to $f$ and write
  $f^h\stto f$. For vector-valued functions, two scale convergence is defined componentwise.
\end{definition}
%
 As stated in the definition we allow only oscillations in variable $x'$. It can be easily seen that this restriction does not influence the main results of two scale convergence. Moreover, since we identify functions on $S$ with their trivial
extension to $\Omega$, the definition above contains the standard
notion of two scale convergence on $S\times Y$ as a special case.

Indeed, when $\{f^h\}_{h>0}$ is a bounded sequence in
$L^2(S)$, then $f^h\wtto f$ is equivalent to
\begin{equation*}
  \lim\limits_{h\to 0}\int_S
  f^h(x')\psi(x',\frac{x'}{\eh})\,dx'=\iint_{S\times Y} f(x',y)\psi(x',y)\,dy\,dx',
\end{equation*}
for all  $\psi\in C^\infty_0(S,C(\mathcal Y))$.

%

%

We have the following characterization of the possible
two scale limits of nonlinear strains.
\begin{proposition}
\label{P:strain}
Let $(u^h)_{h>0}$ be a sequence of deformations with finite bending
energy, let $u\in\WW(S,\R^3)$ with the second fundamental form $\secf$, and assume that
\begin{align*}
  u^h-\fint_\Omega u^h\,dx&\to u,&\qquad&\text{strongly in
  }L^2(\Omega,\R^3),\\
  E^h:=\frac{\sqrt{(\nabla_h u^h)^t\nabla_h u^h}-\id}{h}&\wtto: E,&\qquad&\text{weakly two scale},
\end{align*}
for some $E\in L^2(\Omega\times Y; \M^3)$. Then there exist
  $
  B \in L^2(S, \M^2_{\sym}),
  $
   $\zeta \in L^2(S,\mathring H^1(I \times \mathcal Y,\R^2))$, $\varphi \in L^2(S,\mathring H^2(\mathcal Y))$ and $g\in L^2(S,L^2(I \times \mathcal{Y},\R^3))$ such that
  \begin{equation}\label{P:strain:1a}
    E(x, y) =
        \iota \left(x_3\secf(x')+B(x')\right)
    +\mathcal{U}\big(\zeta(x,\cdot),\varphi(x,\cdot),g(x,\cdot,\cdot)\big) (x_3,y).
  \end{equation}
\end{proposition}

The starting point of the proof of the previous proposition is  \cite[Theorem 6]{FJM-06} (see also the proof of \cite[Theorem 4.1]{FJM-02}. The proof, roughly speaking, consists in dividing $S$ in small cubes of size $h$, applying the theorem on geometric rigidity on each cube to obtain the mapping $R$ from Lemma and then mollifying the mapping $R$ to obtain the mapping $\tilde{R}$.

\begin{lemma}
\label{t6f2}
There exist constants $C,c>0$, depending only
on $S$, such that the following is true:
if $u\in H^1(\Omega, \R^3)$ then there exists a map $R : S\to SO(3)$
which is piecewise constant on each cube $x + hY$ with $x\in h\Z^2$ and there exists
$\t{R}\in H^1(S, \M^3)$ such that for each $\xi \in \R^2$ which satisfy $|\xi|_\infty=\max\{|\xi \cdot e_1|,|\xi \cdot e_2|\}<h$ and for each $S' \subset S$  which satisfy $\dist(S',\partial S)>ch$ we have:
\begin{eqnarray*}
&& \|\nabla_hu - R\|^2_{L^2(S' \times I)} + \|R - \t{R}\|^2_{L^2(S')}+h^2 \|R - \t{R}\|^2_{L^\infty(S')} + h^2\|\nabla'\t{R}\|^2_{L^2(S')}\\
&&+\|\tau_{\xi} R-R\|^2_{L^2(S')} \leq C \|\dist(\nabla_h u, \SO 3) \|^2_{L^2(\Omega)},
\end{eqnarray*}
where $(\tau_{\xi} R)(x'):=R(x'+\xi)$.
\end{lemma}

The following two lemmas give us the characterization of the  two scale limit of the oscillatory part of the sequence $(\nabla' \tilde{R}^h)_{h>0}$. The first lemma uses the fact that it is the sequence which is in $L^2$ norm close to the sequence of rotational field (closer than the order of the oscillations) and thus the oscillatory part of two scale limit (beside being of the gradient form in the fast variable) at every point $x$ is tangential to the rotational limit field. The second lemma uses additional fact  that $\tilde{R}^h$ is close in $L^2$ norm to the gradient field (closer than the order of oscillations) and thus the oscillatory part of two scale limit should be of the form of the second gradient. Combining it with the result of the first lemma we have a specific form of two scale limit of the oscillatory part which is given in the second lemma.
\begin{lemma} \label{rotations}
Let $(R^h)_{h>0} \subset  L^\infty(\tilde{S},\SO 3)$ and $(\t{R}^h)_{h>0} \subset H^1(\tilde{S},\M^3)$ satisfy
for each $h>0$
\begin{equation}\label{eq:0001}
\| \t{R}^h-R^h \|_{L^2} \leq C\eta(h), \ \| \nabla' \t{R}^h \|_{L^2} \leq C, \|\t{R}^h\|_{L^\infty} \leq C,
\end{equation}
where $C>0$ is independent of $h$ and $\eta(h)$ satisfies $\lim_{h \to 0} \tfrac{\eta(h)}{\eh}=0$.
Then for any subsequence of $(\nabla'\t{R}^h)_{h>0}$ which converges two scale there exists a unique $w \in L^2(\tilde{S}, \mathring{H}^1 (\mathcal{Y},\R^3))$ such that
$\partial_\alpha \t{R}^h \wtto \partial_\alpha R+ RA_{\partial_{y_{\alpha}} w(x,y)},$
where $R \in H^1(\tilde{S},\SO 3)$ is the weak limit of $\t{R}^h$.
\end{lemma}
\begin{proof}
We have to prove only that
\begin{equation} \label{eq:0003}
M^h_\alpha:=\sym [(\t{R}^h)^T \partial_\alpha \t{R}^h] \wtto \sym (R^T \partial_\alpha R)=0, \text{ for } \alpha=1,2.
\end{equation}
The rest is a direct consequence of (v) in Lemma \ref{kompts}. Namely, let us assume (\ref{eq:0003}). Then we conclude that for $\alpha=1,2$ there exists $\t{w}_\alpha \in L^2(S,\mathring{H}^1(\mathcal{Y},\R^3))$ such that
\begin{equation}\label{eq:0004}
(\t{R}^h)^T \partial_\alpha \t{R}^h \wtto R^T \partial_\alpha R+A_{\t{w}_\alpha}.
\end{equation}
Using the fact that $\t{R}^h \to R$,  boundedly in measure and Lemma \ref{mnozenjets} we conclude from (\ref{eq:0004}) that
\begin{equation}\label{eq:0005}
\partial_\alpha \t{R}^h \wtto  \partial_\alpha R+RA_{\t{w}_\alpha(x,y)}.
\end{equation}
By using (v) of Lemma \ref{kompts} we can also conclude that there exists $M \in L^2({\tilde{S}},\mathring{H}^1(\mathcal{Y},\M^3))$ such that
\begin{equation}\label{eq:0006}
\partial_\alpha \t{R}^h \wtto  \partial_\alpha R+\partial_{y_\alpha} M.
\end{equation}
From (\ref{eq:0005}) and (\ref{eq:0006}) we conclude that
$$A_{\t{w}_\alpha(x,y)}=\partial_{y_\alpha}(R^T M),$$
which implies that $\t{w}_\alpha=\partial_\alpha w$, where $w$ is the axial vector of $\skw (R^T M)$.

It remains to prove (\ref{eq:0003}). Notice that
\begin{equation} \label{eq:0002}
M^h_{\alpha,ij}=\tfrac{1}{2}\partial_\alpha (\t{R}^h_i \cdot \t{R}^h_j).
\end{equation}
Take $\psi \in  C_0^\infty
(\tilde{S},\mathring{C}^\infty(\mathcal{Y};\R^3))$ and calculate
\begin{eqnarray*}
\lim_{h \to 0} \int_{{\tilde{S}}} M^h_{\alpha,ij} \psi(\cdot,\tfrac{\cdot}{\eh}) \ud x &=& \lim_{h \to 0} \int_{{\tilde{S}}} \tfrac{1}{2} (\t{R}^h_i \cdot \t{R}^h_j) \partial_{x_\alpha} \psi(\cdot,\tfrac{\cdot}{\eh})\\ & &+\lim_{h \to 0} \int_{{\tilde{S}}} \tfrac{1}{2 \eh} (\t{R}^h_i \cdot \t{R}^h_j) \partial_{y_\alpha} \psi(\cdot,\tfrac{\cdot}{\eh})\\ \begin{array}{c} \text{using (\ref{eq:0001}) and} \\ \text{Cauchy inequality} \end{array}  &=& \lim_{h \to 0} \int_{{\tilde{S}}} \tfrac{1}{2} (R^h_i \cdot R^h_j) \partial_{x_\alpha} \psi(\cdot,\tfrac{\cdot}{\eh})\\ & &+\lim_{h \to 0} \int_{{\tilde{S}}} \tfrac{1}{2 \eh} (R^h_i \cdot R^h_j) \partial_{y_\alpha} \psi(\cdot,\tfrac{\cdot}{\eh})\\ &=& \int_{{\tilde{S}}} \partial_\alpha (\delta_{ij}) \psi(\cdot,\tfrac{\cdot}{\eh}) \\ &=& 0.
\end{eqnarray*}
This together with (iii) of Lemma \ref{kompts} implies (\ref{eq:0003}).
\end{proof}
\begin{lemma} \label{glavnapomoc}
Let  $(\overline{u}^h)_{h>0} \subset H^1(\tilde{S},\R^3)$, $(R^h)_{h>0} \subset L^\infty(\tilde{S},\SO 3)$ and $(\t{R}^h)_{h>0} \subset H^1(\tilde{S},\M^3)$ satisfy for each $h>0$:
\begin{eqnarray}\label{eq:0000}
&& \| \nabla' \overline{u}^h-(R^h e_1,R^h e_2)\|_{L^2}+\|R^h-\t{R}^h\|_{L^2}+\| \nabla' \t{R}^h \|_{L^2} \leq C \eta(h), \\ \nonumber
&& \|\t{R}^h\|_{L^\infty} \leq C.
\end{eqnarray}
where $C>0$ is independent of $h$ and $\eta(h)$ satisfies $\lim_{h \to 0} \tfrac{\eta(h)}{\eh}=0$. Then
for any subsequence of $(\nabla'\t{R}^h)_{h>0}$ which converges two scale
there exists a unique $w \in L^2(\tilde{S},\mathring{H}^2(\mathcal{Y}))$ such that
\begin{equation}
\partial_\alpha \t{R}^h \wtto \partial_\alpha R+ R \left( \begin{array}{ccc} 0& 0&-\partial_{y_1 y_\alpha} w \\
0 & 0 & -\partial_{y_2 y_\alpha} w \\ \partial_{y_1 y_\alpha} w & \partial_{y_2 y_\alpha} w & 0
 \end{array} \right),
\end{equation}
where $R$ is the weak limit of $\t{R}^h$ in $H^1$.
\end{lemma}
\begin{proof}
From (\ref{eq:0000}) we conclude that there exists $R:S \to \SO 3$ and $u \in H^2(S,\R^3)$ such that $Re_\alpha=\partial_\alpha u$, for $\alpha=1,2$, $\t{R}^h \rightharpoonup R$ weakly in $H^1$, $\partial_\alpha \overline{u}^h \to Re_\alpha$, $R^h \to R$ strongly in $L^2$. Also from (\ref{eq:0000}) and Lemma \ref{rotations} we have that there exists $\t{w} \in L^2(\tilde{S},\mathring{H}^1(\mathcal{Y},\R^3))$ such that for $\alpha=1,2$
\begin{equation}
\partial_\alpha \t{R}^h \wtto \partial_\alpha R+R A_{\partial_{y_\alpha} \t{w}}.
\end{equation}
Using Lemma \ref{pomoooc} we conclude that there exists $v \in  L^2(\tilde{S},\mathring{H}^2(\mathcal{Y},\R^3))$ such that
\begin{equation}
R\left(\begin{array}{cc} 0 & -\partial_{y_\alpha} \t{w}_3 \\ \partial_{y_\alpha} \t{w}_3 & 0 \\ -\partial_{y_\alpha} \t{w}_2 & \partial_{y_\alpha} \t{w}_1 \end{array} \right) = \left(\begin{array}{cc} \partial_{y_1 y_\alpha} v_1  & \partial_{y_2 y_\alpha} v_1 \\ \partial_{y_1 y_\alpha} v_2  & \partial_{y_2 y_\alpha} v_2 \\ \partial_{y_1 y_\alpha} v_3  & \partial_{y_2 y_\alpha} v_3\end{array} \right).
\end{equation}
By putting $\t{v}=R^T v$ we have that for $\alpha=1,2$
\begin{equation}
\left(\begin{array}{cc} 0 & -\partial_{y_\alpha} \t{w}_3 \\ \partial_{y_\alpha} \t{w}_3 & 0 \\ -\partial_{y_\alpha} \t{w}_2 & \partial_{y_\alpha} \t{w}_1 \end{array} \right) = \left(\begin{array}{cc} \partial_{y_1 y_\alpha} \t{v}_1  & \partial_{y_2 y_\alpha} \t{v}_1 \\ \partial_{y_1 y_\alpha} \t{v}_2  & \partial_{y_2 y_\alpha} \t{v}_2 \\ \partial_{y_1 y_\alpha} \t{v}_3  & \partial_{y_2 y_\alpha} \t{v}_3\end{array} \right).
\end{equation}
From this one easily concludes that $\t{w}_3=0$ which implies the claim by defining $w=\t{v}_3$.
\end{proof}
The following lemma gives us criterium how to recognize oscillations which have the form of symmetric gradient field. It is enough to test it with the specific test functions from the orthogonal complement (whose linear combinations are dense in the orthogonal complement).
This lemma was already used in \cite{Hornungvel12} and can be easily proved by Fourier transform. Here we use it to prove Lemma \ref{lema:glavna1}.
\begin{lemma} \label{lm:1}
Let  $M \in L^2(\tilde{S};L^2(\mathcal{Y},\M^2_{\sym})$ such that
for every
$$
\Psi \in C_0^{\infty}(\tilde{S}, C^{\infty}(\mathcal{Y};\M^{2}_{\sym})),
$$
which satisfies
\begin{equation}
\label{eq:iggg10}\Psi(\cdot,y)=(\cof \nabla^2 F)(y) \psi (\cdot),
\end{equation}
for some $\psi \in C_0^\infty(\tilde{S})$, $F\in
\mathring{C}^\infty(\mathcal{Y})$,
we have that
$$\iint_{\tilde{S} \times \mathcal{Y}} M(\cdot,y): \Psi(\cdot,y)=0.$$
Then there exist unique  $M_0 \in L^2(\tilde{S},\M^2_{\sym})$ and
$\zeta \in L^2(\tilde{S},\mathring{H}^1(\mathcal{Y};\R^2))$ such that
$$M=M_0+\sym \nabla_y \zeta. $$
\end{lemma}
The following lemma is crucial for the proof of the Proposition \ref{P:strain}.
\begin{lemma} \label{lema:glavna1}
Let Assumption \ref{assumption:gamma} be satisfied. Let $(\t{u}^h)_{h>0} \subset H^2(\tilde{S},\R^3)$, $(\t{R}^h)_{h>0} \subset H^1(\tilde{S},\M^3)$ and $(R^h)_{h>0} \subset L^\infty(\tilde{S},\SO 3)$ such that for each $h>0$
$R^h$ is piecewise constant on each cube $x + hY$ with $x\in h\Z^2$ and for each $\xi \in \R^2$ which satisfy $|\xi|_\infty=\max\{|\xi \cdot e_1|,|\xi \cdot e_2|\}<h$ we have
\begin{eqnarray}\label{assigor}
&& h^2\|\nabla'^2 \t{u}^h\|^2_{L^2(\tilde{S})}+\|\nabla' \t{u}^h - (R^h e_1, R^h e_2)\|^2_{L^2(\tilde{S})} + \|R^h - \t{R}^h\|^2_{L^2(\tilde{S})}\\ \nonumber
&&+h^2 \|R^h - \t{R}^h\|^2_{L^\infty(\tilde{S})} + h^2\|\nabla'\t{R}^h\|^2_{L^2(\tilde{S})}+\|\tau_{\xi} R^h-R^h\|^2_{L^2(\tilde{S}^h)} \leq Ch^2,
\end{eqnarray}
for some $C>0$ and for each  sequence of subdomains $\tilde{S}^h \subset \tilde{S}$ which satisfy $\dist(\tilde{S}^h,\partial \tilde{S}) \geq ch$ for some $c>0$.
Then  there exist $M_0 \in L^2(\tilde{S},\M^2_{\sym})$ and  $\zeta \in L^2(\tilde{S},\mathring{H}^1(\mathcal{Y},\R^2))$ such that for $\alpha,\beta=1,2$ (on a subsequence) we have:
\begin{eqnarray*}
  M^h_{\alpha \beta} &:=& \tfrac{1}{2h}[(R^he_\alpha) \cdot  \partial_\beta \t{u}^h+(R^he_\beta) \cdot  \partial_\alpha \t{u}^h]-\delta_{\alpha \beta}\\ &\wtto& M_{0,\alpha \beta} +\tfrac{1}{2}(\partial_{y_\alpha} \zeta_{\beta}+\partial_{y_\beta}\zeta_\alpha).
 \end{eqnarray*}
\end{lemma}
\begin{proof}
From (\ref{assigor}) we can assume that there exists $u \in H^2(\tilde{S},\R^3)$ and $R \in H^1(\tilde{S},\SO 3)$ such that
$\partial_\alpha u= Re_\alpha$, $\partial_\alpha \t{u}^h \rightharpoonup \partial_\alpha u$ weakly in $H^1$ for $\alpha=1,2$, $\t{R}^h \rightharpoonup R$ weakly in $H^1$ and $R^h \to R$ strongly in $L^2$. Let us suppose that $M^h \wtto M$, for some $M \in L^2(\tilde{S} \times \mathcal{Y},\M^2)$. Using Lemma \ref{lm:1} it is enough to see
that
$$\iint_{\tilde{S} \times Y} M(\cdot,y): \Psi(\cdot,y) =0,$$
where $\Psi$ is defined by (\ref{eq:iggg10}). We have
\begin{align*}
& \iint_{\tilde{S} \times Y} M(\cdot,y): \Psi (\cdot,y)
\\
&=\lim_{h \to 0} \int_{\tilde{S}} M^h: (\cof \nabla^2 F)\left( \tfrac{\cdot}{\eh} \right) \psi
\\
&=\lim_{h \to 0} \tfrac{\eh^2}{h}\int_{\tilde{S}}  hM^h : \cof
\left[
\nabla^2 \left( F\left( \tfrac{\cdot}{\eh} \right) \psi\right)
- 2 \nabla \left(F\left( \tfrac{\cdot}{\eh} \right) \nabla\psi \right) + F\left( \tfrac{\cdot}{\eh}\right) \nabla^2\psi
\right] \\
&=\lim_{h \to 0} \tfrac{\eh^2}{h}\int_{\tilde{S}}  hM^h : \cof
\left[
\nabla^2 \left( F\left( \tfrac{\cdot}{\eh} \right) \psi\right)
- 2 \nabla \left(F\left( \tfrac{\cdot}{\eh} \right) \nabla\psi \right)\right].
\end{align*}
It is easy to conclude that
$$ \lim_{h\to 0} \tfrac{\eh^2}{h}\int_{\tilde{S}}  hM^h : \cof\left[ \nabla \left(F\left( \tfrac{\cdot}{\eh} \right) \nabla\psi \right)\right]=0. $$
Namely, it is enough to conclude that the sequence $$I^h:=\int_{\tilde{S}}  hM^h : \cof\left[ \nabla \left(F\left( \tfrac{\cdot}{\eh} \right) \nabla\psi \right)\right],$$ is bounded. To see this notice that, because of (\ref{assigor}), we have $|I^h-\t{I}^h| \to 0$, where
\begin{equation} \label{eq:0010}
 \t{I}^h:=\int_{\tilde{S}}  M_c^h : \cof\left[ \nabla \left(F\left( \tfrac{\cdot}{\eh} \right) \nabla\psi \right)\right],
\end{equation}
and
$$ M_c^h:=\tfrac{1}{2}[(\t{R}^he_\alpha) \cdot  \partial_\beta \t{u}^h+(\t{R}^he_\beta) \cdot  \partial_\alpha \t{u}^h]. $$
By  partial integration in (\ref{eq:0010}) and the fact that $\|\nabla M_c^h\|_{L^1(\tilde{S})}$ is bounded we easily obtain the boundedness of $\t{I}^h$. From this it follows the boundedness of $I^h$.
It remains to prove that
\begin{equation}
\lim_{h \to 0} \tfrac{\eh^2}{h}\int_{\tilde{S}}  hM^h : \cof
\nabla^2 \left( F\left( \tfrac{\cdot}{\eh} \right) \psi\right)=0.
\end{equation}
By partial integration we obtain for $h$ small enough:
\begin{align}
& \int_{\tilde{S}}  hM^h : \cof
\nabla^2 \left( F\left( \tfrac{\cdot}{\eh} \right) \psi\right)= 
\\ \nonumber &
\int_{\tilde{S}} (R^he_2) \cdot \partial_{11} \t{u}^h \partial_2 \left( F\left( \tfrac{\cdot}{\eh} \right) \psi\right)
-\int_{\tilde{S}} (R^he_2) \cdot \partial_{12} \t{u}^h \partial_1 \left( F\left( \tfrac{\cdot}{\eh} \right) \psi\right)
\\ \nonumber & +
\sum_{x' \in h\Z^2 \cap {\tilde{S}}^h}\left( (R^he_1)(x')-(R^he_1)(x'+h(0,1))\right) \cdot \left( \int_{\Gamma^{2,h}_{x'}} \partial_1 \t{u}^h \partial_{2} \left(F\left( \tfrac{\cdot}{\eh} \right) \psi\right) \right)
\end{align}
\begin{align*}
 \nonumber & -
\sum_{x' \in h\Z^2 \cap {\tilde{S}}^h}\left( (R^he_1)(x')-(R^he_1)(x'+h(1,0))\right) \cdot \left( \int_{\Gamma^{1,h}_{x'}} \partial_2 \t{u}^h \partial_{2} \left(F\left( \tfrac{\cdot}{\eh} \right) \psi\right) \right)\\
\nonumber & -
\sum_{x' \in h\Z^2 \cap {\tilde{S}}^h}\left( (R^he_2)(x')-(R^he_2)(x'+h(1,0))\right) \cdot \left( \int_{\Gamma^{1,h}_{x'}} \partial_1 \t{u}^h \partial_{2} \left(F\left( \tfrac{\cdot}{\eh} \right) \psi\right) \right)
\\
\nonumber & +
\sum_{x' \in h\Z^2 \cap {\tilde{S}}^h}\left( (R^he_2)(x')-(R^he_2)(x'+h(1,0))\right) \cdot \left( \int_{\Gamma^{1,h}_{x'}} \partial_2 \t{u}^h \partial_{1} \left(F\left( \tfrac{\cdot}{\eh} \right) \psi\right) \right),
\end{align*}
where $\Gamma^{1,h}_{x'}$ is the segment $[x'+h(1,0),x'+h(1,1)]$, $\Gamma^{2,h}_{x'}$ is the segment $[x'+h(0,1),x'+h(1,1)]$ and ${\tilde{S}}^h$ is a compact subset of ${\tilde{S}}$ such that $\spt \psi \subset {\tilde{S}}^h$. First we will prove that $\lim_{h \to 0} \tfrac{\eh^2}{h} I_1^h=0$, where
$$ I_1^h=\int_{\tilde{S}} (R^he_2) \cdot \partial_{11} \t{u}^h \partial_2 \left( F\left( \tfrac{\cdot}{\eh} \right) \psi\right)
-\int_{\tilde{S}} (R^he_2) \cdot \partial_{12} \t{u}^h \partial_1 \left( F\left( \tfrac{\cdot}{\eh} \right) \psi\right).$$
To prove this it is enough to prove the boundedness of  the sequence $I_1^h$.
Notice, as before, using (\ref{assigor}) and Cauchy inequality that $|I_1^h-\t{I}_1^h| \to 0$, where
$$ \t{I}_1^h=\int_{\tilde{S}} (\t{R}^he_2) \cdot \partial_{11} \t{u}^h \partial_2 \left( F\left( \tfrac{\cdot}{\eh} \right) \psi\right)
-\int_{\tilde{S}} (\t{R}^he_2) \cdot \partial_{12} \t{u}^h \partial_1 \left( F\left( \tfrac{\cdot}{\eh} \right) \psi\right).$$
By replacing $\t{u}^h$ with a smooth function $\t{u}^h_c \in C^3({\tilde{S}})$ such that $\|\t{u}^h-\t{u}^h_c\|_{H^2} \ll \eh$ we obtain, after partial integration, that $| \t{I}_1^h-\t{I}_{1,c}^h | \to 0$, where
$$ \t{I}_{1,c}^h=-\int_{\tilde{S}} \partial_2 (\t{R}^he_2) \cdot \partial_{11} \t{u}_c^h   F\left( \tfrac{\cdot}{\eh} \right) \psi
+\int_{\tilde{S}} \partial_1 (\t{R}^he_2) \cdot \partial_{12} \t{u}^h_c  F\left( \tfrac{\cdot}{\eh} \right) \psi.$$
Now we easily obtain the boundedness of $\t{I}_{1,c}^h$ which implies the boundedness of $I_1^h$.
We want to show that $\lim_{h \to 0} \tfrac{\eh^2}{h} I_2^h=0$, where
\begin{align*}
 I_2^h:= &
\sum_{x' \in h\Z^2 \cap {\tilde{S}}^h}\left( (R^he_1)(x')-(R^he_1)(x'+h(0,1))\right) \cdot \\ & \left( \int_{\Gamma^{2,h}_{x'}} \partial_1 \t{u}^h \partial_{2} \left(F\left( \tfrac{\cdot}{\eh} \right) \psi\right) \right).
\end{align*}
We will prove even more that $I_2^h \to 0$.
Since $\t{u}^h \in H^2(\Omega_{x'}^h,\R^3)$ we have that
\begin{equation} \label{eq:0020}
\int_{\Gamma^{2,h}_{x'}} \left|\partial_1 \t{u}^h-\tfrac{1}{h} \int_{\Gamma_{\cdot}^{1,h}} \partial_1 \t{u}^h \right|^2 \leq \tfrac{h}{3} \int_{\Omega_{x'}^h} |\partial_{12} \t{u}^h|^2,
\end{equation}
where for $x\in \Gamma_{x'}^{2,h}$ we put $\Gamma_{x}^{1,h}=[x-h(0,1),x]$ and $\Omega_{x'}^h$ is the square of side $h$ whose down left corner is $x'$. From (\ref{assigor}) we easily conclude that for $\alpha=1,2$ and $\xi \in \R^2$, $|\xi|_\infty=1$ we have
\begin{equation} \label{eq:0021}
\sum_{x' \in h\Z^2 \cap {\tilde{S}}^h}\left( (R^he_\alpha)(x')-(R^he_\alpha)(x'+h\xi)\right)^2 \leq C,
\end{equation}
Using Cauchy inequality and (\ref{eq:0020}), (\ref{eq:0021}) we conclude that $|I_2^h-\t{I}_2^h| \to 0$, where
 \begin{align*}
  \t{I}_2^h:= & \sum_{x' \in h\Z^2 \cap {\tilde{S}}^h}\left( (R^he_1)(x')-(R^he_1)(x'+h(0,1))\right) \cdot  \\ & \hspace{5ex} \left( \tfrac{1}{h}\int_{\Omega^h_{x'}} \partial_1 \t{u}^h \partial_{2} \left(F\left( \tfrac{\cdot}{\eh} \right) \psi\right)\circ P^h_{x'} \right),
 \end{align*}
and $P^h_{x'}:\Omega^h_{x'} \to \Gamma^{2,h}_{x'}$ is the projection. From (\ref{assigor}) and Cauchy inequality we can easily conclude that $|\t{I}^h_2-\t{I}^h_{2,c}| \to 0$, where
 \begin{align*}
 \t{I}^h_{2,c} :=&
 \sum_{x' \in h\Z^2 \cap {\tilde{S}}^h}\left( (R^he_1)(x')-(R^he_1)(x'+h(0,1))\right) \cdot \\ & \hspace{5ex}(R^he_1)(x') \int_{\Gamma^{2,h}_{x'}} \partial_2 \left(F\left( \tfrac{\cdot}{\eh} \right) \psi\right).
 \end{align*}
By using (\ref{eq:0021}) we easily obtain that
$|\t{I}^h_{2,c}-\t{I}^h_{2,cc}| \to 0$, where
 \begin{align*}
 \t{I}^h_{2,cc} :=&
 \sum_{x' \in h\Z^2 \cap {\tilde{S}}^h}\left( (R^he_1)(x')-(R^he_1)(x'+h(0,1))\right) \cdot \\ & \tfrac{1}{2}\left((R^he_1)(x')+(R^he_1)(x'+h(0,1))\right)  \int_{\Gamma^{2,h}_{x'}} \partial_2 \left(F\left( \tfrac{\cdot}{\eh} \right) \psi\right)=0.
 \end{align*}
This implies that $I_2^{h} \to 0$.
In the same way we can conclude that $I_4^h \to 0$, where
\begin{align*}
 I_4^h:=& \sum_{x' \in h\Z^2 \cap {\tilde{S}}^h}\left( (R^he_2)(x')-(R^he_2)(x'+h(1,0))\right) \cdot \\ & \hspace{5ex}  \left( \int_{\Gamma^{1,h}_{x'}} \partial_2 \t{u}^h \partial_{1} \left(F\left( \tfrac{\cdot}{\eh} \right) \psi\right) \right).
\end{align*}
It remains to check the part
\begin{align*}
I_3^h :=&
\sum_{x' \in h\Z^2 \cap {\tilde{S}}^h}\left( (R^he_1)(x')-(R^he_1)(x'+h(1,0))\right) \cdot \\ & \hspace{5ex}\left( \int_{\Gamma^{1,h}_{x'}} \partial_2 \t{u}^h \partial_{2} \left(F\left( \tfrac{\cdot}{\eh} \right) \psi\right) \right)\\
\nonumber & -
\sum_{x' \in h\Z^2 \cap {\tilde{S}}^h}\left( (R^he_2)(x')-(R^he_2)(x'+h(1,0))\right) \cdot \\ & \hspace{5ex} \left( \int_{\Gamma^{1,h}_{x'}} \partial_1 \t{u}^h \partial_{2} \left(F\left( \tfrac{\cdot}{\eh} \right) \psi\right) \right).
\end{align*}
 We follow the same pattern to replace $\partial_1 \t{u}^h$  ($\partial_2 \t{u}^h$) by $R^h e_1$ ($R^h e_2$) and obtain that $|I_3^h-I_{3,c}^h| \to 0$, where
\begin{align*}
I_{3,c}^h :=&
\sum_{x' \in h\Z^2 \cap {\tilde{S}}^h}\left( (R^he_1)(x')-(R^he_1)(x'+h(1,0))\right) \cdot \\ & \hspace{5ex} (R^h e_2) (x')\int_{\Gamma^{1,h}_{x'}} \partial_{2} \left(F\left( \tfrac{\cdot}{\eh} \right) \psi\right)\\
\nonumber & +
\sum_{x' \in h\Z^2 \cap {\tilde{S}}^h}\left( (R^he_2)(x')-(R^he_2)(x'+h(1,0))\right) \cdot \\ & \hspace{5ex} (R^h e_1) (x') \int_{\Gamma^{1,h}_{x'}} \partial_{2} \left(F\left( \tfrac{\cdot}{\eh} \right) \psi \right).
\end{align*}
Using again (\ref{eq:0021}) we easily obtain that $|I_{3,c}^h-I_{3,cc}^h| \to 0$, where
\begin{eqnarray*}
I_{3,cc}^h &:=&
\sum_{x' \in h\Z^2 \cap {\tilde{S}}^h}\left( (R^he_1)(x')-(R^he_1)(x'+h(1,0))\right) \cdot \\ & & \hspace{5ex}\tfrac{1}{2} \left((R^h e_2) (x')+(R^he_2)(x'+h(1,0))\right)\int_{\Gamma^{1,h}_{x'}} \partial_{2} \left(F\left( \tfrac{\cdot}{\eh} \right) \psi\right)\\
\nonumber & &+
\sum_{x' \in h\Z^2 \cap {\tilde{S}}^h}\left( (R^he_2)(x')-(R^he_2)(x'+h(1,0))\right) \cdot \\ & &\hspace{5ex}
\tfrac{1}{2} \left((R^h e_1) (x')+(R^he_1)(x'+h(1,0))\right)
 \int_{\Gamma^{1,h}_{x'}} \partial_{2} \left(F\left( \tfrac{\cdot}{\eh} \right) \psi \right)\\ &=& 0.
\end{eqnarray*}
This finishes the proof of the claim.
\end{proof}
\begin{remark}
\label{slucaj1} In the case $\eh^2 \sim h$ the major difficulty is to deal with the expression $\t{I}^h_{1,c}$, since under the integral we have the product of two sequences which are only bounded in $L^2$.
\end{remark}

For proving the Proposition~\ref{P:strain}  we need the following remark, the same as we needed in the proof of \cite[Proposition~3.1]{NeuVel-12}. However here we apply it in completely different situation, to correct the minimizing sequence (averaged over the thickness) to obtain the sequence bounded in $H^2$.
\begin{remark}
  \label{rem:regularity}
Assume that $\t{S} \subset \R^2$  a bounded domain. Let us look the following minimization problem
  \begin{equation*}
    \min_{v\in H^1(\t S)\atop\int_{\t S} v=0}\int_{\t S}|\nabla v-p|^2\,d x',
 \end{equation*}
  where $p \in H^1(\t{S},\R^2)$ is  a given field. The
  associated Euler-Lagrange equation reads
  \begin{equation*}
    \left\{\begin{aligned}
      -\triangle v=\,&\,-\nabla\cdot p \qquad&&\text{in } \t S\\
      \partial_\nu v=\,&\,p \cdot\nu\qquad&&\text{on }\partial \t S,
    \end{aligned}\right.
  \end{equation*}
 subject to $\int_{\t{S}} v\,dx=0$. Above, $\nu$ denotes the normal on
 $\partial \t{S}$. Since $\nabla\cdot p \in L^2$, we obtain by
 standard  regularity estimates that  $v\in H^2 (\t S)$ under the assumption that $\partial\t S$ is $C^{1,1}$. In this case it holds
  $\|v\|_{H^2(\t S)}\lesssim \|\nabla\cdot
  p\|_{L^2(\t S)}+\|p\|_{L^2(\t S)}$.
\end{remark}

\begin{proof}[Proof of Proposition~\ref{P:strain}]
Without loss of generality we assume that all $u^h$ have average zero.
Take $\t{S} \subset S$, open and compactly contained in $S$, such that $\partial \t{S}$ is of class $C^{1,1}$
and  define $\t{\Omega}:=\t{S} \times I$.
Let $R^h$, $\t{R}^h$ be the maps obtained by applying Lemma~\ref{t6f2}
to $u^h$. Due to the uniform bound on $\nabla'\t R^h$ given by
Lemma~\ref{t6f2},  $R^h$ and $\t R^h$ are precompact in
$L^2(\tilde{S},\M^3)$.
Hence,   $R^h$ and $\t R^h$ strongly converge in
$L^2(\tilde{S},\M^3)$ to $R \in H^1(\tilde{S},\SO 3)$ on a subsequence. Also we can conclude that
$u^h \to u$ strongly in $H^1(\tilde{\Omega},\R^3)$ and
  $\nabla_hu^h\to R=(\nabla'u,\ n)$ strongly in
$L^2(\tilde{\Omega},\M^3)$, for $u \in H^2(\tilde{S},\R^3)$.
Define $\overline{u}^h(x')=\int_{I} u^h(x',x_3) \ud x_3$ and notice that
\begin{equation}
\|\nabla' \overline{u}^h-(\t{R}^h e_1, \t{R}^h e_2) \|^2_{L^2(\tilde{S})} \leq Ch^2,
\end{equation}
for some $C>0$. Define $\t{u}^h \in H^2(\t{S},\R^3)$ such that
$\t{u}^h$ minimizes the problem
 \begin{equation*}
    \min_{v\in H^1(\t S,\R^3)\atop\int_{\t{S}} v=0}\int_{\t S}|\nabla v-(\t{R}^he_1,\t{R}^he_2) |^2\,d x'.
 \end{equation*}
From Remark \ref{rem:regularity} we conclude that there exists $C>0$ such that
\begin{eqnarray*}
& & \|\t{u}^h\|_{H^2(\t{S})} \leq C, \quad \|\nabla' \t{u}^h-(\t{R}^he_1, \t{R}^h e_2)\|^2_{L^2(\t S)} \leq Ch^2, \\ & & \quad \|\nabla' \t{u}^h-\nabla'\overline{u}^h\|^2_{L^2(\t S)} \leq Ch^2.
\end{eqnarray*}
Let us suppose that on a subsequence
\begin{equation}\label{kon0}
\tfrac{1}{h} (R^h)^t(\nabla' \overline{u}^h - \nabla' \t{u}^h) \wtto  \theta(x')+\nabla_y v(x',y),
\end{equation}
where $\theta \in L^2(\t{S},\M^{3\times 2})$ and $v \in L^2(\t{S},\mathring{H}^1(\mathcal{Y},\R^3))$.
This can be done without loss of generality by using Lemma \ref{kompts}, Lemma \ref{mnozenjets} and Poincare inequality.
Following \cite{FJM-02}, we introduce the approximate
 strain
\begin{equation}
  G^h(x)=\frac{(R^h)^t \nabla_h u^h(x)-\id}{h}.
\end{equation}
Since $G^h$ is bounded in $L^2(\tilde{\Omega})$ we can assume that
$G^h \wtto G \in L^2(\tilde{\Omega} \times \mathcal{Y},\M^3)$.
First notice that it suffices to identify the
symmetric part of the two scale limit $G$ of the sequence $G^h$. Indeed, since
$\sqrt{(\id+h F)^t(\id +h F)}=\id + h\sym F$ up to terms of higher
order, the convergence $G^h\wtto G$
implies $E=\sym G$ (see e.g. \cite[Lemma~4.4]{Neukamm-11} for a
proof).
We define $z^h\in H^1(\tilde{\Omega}, \R^3)$ via
\begin{equation}
  u^h(x', x_3) = \overline{u}^h(x') + hx_3 \t{R}^h(x') e_3 + h z^h(x', x_3).
\end{equation}
Then clearly $\int_I z^h(x', x_3) dx_3 = 0$ and we compute
\begin{eqnarray}
& &  \label{combi-0} \\ \nonumber
  G^h&=&
  \frac{ \iota\left( (R^he_1\, ,\, R^h e_2)^t \nabla'\t{u}^h - I' \right)}{h}+\frac{1}{h}\sum_{\alpha=1,2}(R^h e_3 \cdot \partial_\alpha \t{u}^h) e_3 \otimes e_\alpha \\ \nonumber & &+\frac{1}{h} (R^h)^t(\nabla' \overline{u}^h-\nabla'\t{u}^h\, ,\,0) +\frac{1}{h}\left((R^h)^t \t{R}^h e_3 \otimes e_3-(0\, ,0 \, ,e_3)\right)\\ \nonumber & & + x_3(R^h)^t (\nabla'\t{R}^he_3 \,, \, 0) +(R^h)^t \nabla_h z^h,
\end{eqnarray}
where by $I'$ we have denoted the unit matrix in $\M^2$.
By using Lemma \ref{lema:glavna1} we conclude that there exist $\t{B} \in L^2(\t{S},\M^{2}_{\sym})$, $\t{\zeta} \in L^2(\t{S},\mathring{H}^1(\mathcal{Y},\R^2))$ such that on a subsequence
\begin{equation} \label{kon1}
\sym \frac{\left( (R^he_1\;,\; R^h e_2)^t \nabla'\t{u}^h - I' \right)}{h} \wtto \t{B}(x')+\sym \nabla_y \t{\zeta}(x',y).
\end{equation}
Using Lemma \ref{mnozenjets} and Lemma \ref{glavnapomoc} we conclude that there exists
$\varphi \in L^2(\tilde{S},\mathring{H}^2(\mathcal{Y}))$ such that
\begin{equation} \label{kon2}
 x_3(R^h)^t (\nabla'\t{R}^he_3) \wtto x_3 \left(\secf(x')\ , \ 0\right)^t+  x_3 \left(\nabla^2_y \varphi(x',y) \ ,\ 0\right)^t.
\end{equation}
Using Lemma \ref{L:two scale} and Lemma \ref{mnozenjets} and the fact that $\int_I z^h=0$ we conclude that there exists $\phi \in $ $L^2(\tilde{S},$ $\mathring{H}^1(\mathcal{Y},\R^3))$, $d \in L^2(\tilde{S},L^2(I \times \mathcal{Y},\R^3))$ such that on a subsequence
\begin{equation}\label{kon3}
(R^h)^t \nabla_h z^h \wtto  (\nabla_y\phi\,,\,d).
\end{equation}
Without loss of generality we can assume that there exist $p \in L^2(\t{S}  \times \mathcal{Y},\R^2)$ and $z \in L^2(\tilde{S} \times \mathcal{Y};\R^3)$ such that
\begin{eqnarray} \label{kon4}
R^h e_3 \cdot \partial_\alpha \t{u}^h & \wtto& p_\alpha(x',y), \quad \text{for } \alpha=1,2,\\
\label{kon5} \frac{1}{h}\left((R^h)^t \t{R}^h e_3 -e_3\right) &\wtto& z(x',y).
\end{eqnarray}
Using (\ref{kon0}) as well as (\ref{kon1})-(\ref{kon4}) we conclude that there exists
$\zeta \in L^2(\t{\Omega},\mathring{H}^1(\mathcal{Y},\R^2))$, $\varphi \in L^2(\t{S},\mathring{H}^2(\mathcal{Y}))$ and
$g \in L^2(\t{\Omega} \times \mathcal{Y} ,\R^3)$ such that
\begin{eqnarray} \label{kon10}
    E(x, y) &=&
    \iota \left(x_3\secf(x')+B(x') \right)
    +\mathcal{U}\big(\zeta(x,\cdot),\varphi(x,\cdot),g(x,\cdot,\cdot)\big) (x_3,y),\\ \nonumber
&& \forall (x,y) \in \t{\Omega} \times \mathcal{Y},
\end{eqnarray}
where
\begin{eqnarray*}
B &=& \t{B}+\sym \sum_{\alpha,\beta=1,2} \theta_{\alpha \beta} e_\alpha \otimes e_\beta,\\
\zeta &=& \t{\zeta}+\sum_{\alpha=1,2} (\phi_\alpha+v_\alpha)e_\alpha,\\
g &=& \tfrac{1}{2}\sum_{\alpha=1,2} (\partial_{y_\alpha}\phi_3+\partial_{y_\alpha}v_3
+d_\alpha+\theta_{3\alpha}+p_\alpha+z_\alpha) e_\alpha+(d_3+z_3)e_3.\\
\end{eqnarray*}
To obtain the representation (\ref{kon10}) for all $(x,y) \in \Omega \times \mathcal{Y}$ and some
$\zeta \in L^2(\Omega,\mathring{H}^1(\mathcal{Y},\R^2)$, $\varphi \in L^2(S,\mathring{H}^2(\mathcal{Y}))$ and
$g \in L^2(\Omega \times \mathcal{Y} ,\R^3)$
 it is enough to use that  $E \in L^2(\Omega \times \mathcal{Y},\R^{3 \times 3})$ and to exhaust $S$ by an increasing sequence $(\t{S}_n)_{n \in \mathbb{N}} $ of the sets with $C^{1,1}$ boundary and using Remark \ref{kon20}.
\end{proof}

\section{Proof of Theorem~\ref{T:main}}\label{S:proof}

\subsection{Lower bound}

As a preliminary step we need to establish some continuity properties of the quadratic form appearing
in \eqref{ass:expansion} and its relaxed version introduced in Definition~\ref{D:relaxation-formula}.
The proof of the following lemma is direct and we refer  to \cite[Lemma~2.7]{Neukamm-11}.

\begin{lemma}
  \label{lem:1}
  Let $W$ be as in Assumption~\ref{assumption} and let $Q$ be the quadratic form
  associated with $W$ via \eqref{ass:expansion}. Then
  \begin{enumerate}[(i)]
  \item[(Q1)] $Q(\cdot,G)$ is $Y$-periodic and
      measurable for all $G\in\M^3$,
  \item[(Q2)] for  almost every $ (x_3,y)\in \R \times \R^2$ the map
      $ Q( x_3 , y,\cdot)$ is quadratic and satisfies
    \begin{equation}\label{Qelliptisch}
      c_1|\sym G|^2\leq Q(x_3,y,G)=Q(x_3,y,\sym G)\leq c_2|\sym G|^2\ \forall G\in\R^{3\times 3}.
    \end{equation}
  \end{enumerate}
\end{lemma}

The following lemma characterizes the limit energy density.
\begin{lemma}\label{L:Qgamma}
    For all $ A \in \M^{2}_{\sym}$  there exist a unique quadraple $(B,\zeta,\varphi,g)$ with  $B \in  \M^2_{\sym}$
    and $\zeta \in \mathring H^1(\mathcal Y,\R^2)$, $\varphi \in \mathring{H}^2(\mathcal{Y})$, $g \in L^2(I \times \mathcal{Y},\R^3)$ such that:
  \begin{equation*}
    Q_{2}^{rel}(A)=\iint_{I\times
      Y}Q\left( x_3,y,\;\iota(x_3 A+ B)+\mathcal{U}(\zeta,\varphi,g) \,\right) dydx_3.
  \end{equation*}
  The induced mapping $\M^2_{\sym}\ni A\mapsto (B,\zeta,\varphi,g)\in
  \M^2_{\sym}\times\mathring H^1(\mathcal Y,\R^2)\times \mathring{H}^2(\mathcal{Y}) \times L^2(I\times \mathcal{Y}, \R^3)$ is
  bounded and linear and thus $ A\mapsto Q_{2}^{\text{rel}}(A)$ is quadratic.
 \end{lemma}
\begin{proof} 
By \eqref{Qelliptisch} and by Remark \ref{kon20} for each $A\in\M^{2}_{\sym}$
the bilinear form associated with the quadratic functional 
$$
G \mapsto \int_{Y\times I} Q(x_3,y, x_3 A + G )\ dy dx_3.
$$
is elliptic  on the closed linear subspace of $L^2(I \times \mathcal{Y},\M^{3}_{\sym})$ given by
$$
X := \iota(\M^2_{\sym})+L_0(I\times \mathcal{Y},\M^3_{\sym}).
$$
Hence it admits a unique minimizer $ G_0 \in X$ by Riesz representation theorem. Linearity of $G_0$ in $A$ follows immediately from that.
\end{proof}

\begin{remark} \label{rem:jednostavno}
It can be easily seen that
\begin{eqnarray*}
  Q_{2}^{\text{rel}}(A)&:=&
  \inf_{B,\zeta,\varphi}\iint_{I\times Y}Q_2\left( x_3,
    y,\;x_3A+B+\sym\nabla_y\zeta+x_3\nabla_y^2\varphi\right)\,dy\,dx_3,
\end{eqnarray*}
where the infimum is taken over all $B\in\M^{2}_{\sym}$,
$\zeta\in \mathring{H}^1(\mathcal Y,\R^2)$ and $\varphi\in \mathring{H}^2(\mathcal Y)$.
 In the case when $W$ and consequently $Q$ are independent of $x_3$, but still dependent on $y$, the relaxation formula looks significantly simpler i.e. the minimizing $B$ and $\zeta$ in the above expression are easily shown to be $0$. Thus the homogenization of the layered materials is more complex than the materials that have energy density independent of $x_3$.
\end{remark}
\begin{remark} \label{isotropicformula}
Here we discuss the cell formula for the case of isotropic energy and layered material, homogeneous in $x_3$ direction.
Then for $A \in M^3_{\sym}$ we have
$$Q(y,A)=\mu(y)|A|^2+\frac{\lambda(y)}{2}(\tr(A))^2,$$
for some functions $\mu,\lambda:\mathcal{Y} \to \R$ (see e.g. \cite{CiarletI}), which have to satisfy some additional properties to keep the form $Q(y,\cdot)$ positive definite.
The quadratic functional $Q_2$ can then be easily calculated and is, for $A \in M^2_{\sym}$, given by
\begin{equation}
Q_2(y,A)=\tilde{\mu}(y)|A|^2+\frac{\tilde{\lambda}(y)}{2}(\tr(A))^2,
\end{equation}
where
$$ \tilde{\mu}(y)=\mu(y), \quad \tilde{\lambda}(y)=\frac{\mu(y) \lambda(y)}{ 2\mu(y)+\lambda(y)}.$$
Since $Q_2$ does not depend on $x_3$ we have, according to the previous remark,
$$Q_{2}^{\text{rel}}(A)=
  \frac{1}{12}\inf_{\varphi \in \mathring{H}^2(\mathcal{Y})}\int_{Y}Q_2\left(y,\;A+\nabla_y^2\varphi\right)\,dy. $$
\end{remark}
In the case when we have layered material, i.e., only oscillations in one direction minimization formula can be easily computed
\begin{eqnarray*}
Q_{2}^{\text{rel}}(A)&=&
  \frac{1}{12}\inf_{\varphi \in \mathring{L}^2([0,1])}\int_{[0,1]}\Big(\tilde{\mu}(y)\left((A_{11}+\varphi)^2+2A_{12}^2+A_{22}^2\right)\\
  & &+ \frac{\tilde{\lambda}(y)}{2}(A_{11}+A_{22}+\varphi)^2\Big)\,dy.
\end{eqnarray*}
By a straightforward calculation it can be easily seen that
$$Q_{2}^{\text{rel}}(A)=\frac{1}{12} \left(c_1 A_{11}^2+c_2 A_{22}^2+2\langle \tilde{\mu}\rangle A_{12}^2+c_3 A_{11}A_{22}\right), $$
where
$$
c_1=\left\langle\tilde{\mu}+\frac{\tilde{\lambda}}{2}\right\rangle_H, \quad c_2=\left\langle\frac{\tilde{\lambda}}{2}\right\rangle_H +\langle \tilde{\mu} \rangle,
$$
$$
c_3=\left\langle\tilde{\mu}+\frac{\tilde{\lambda}}{2}\right\rangle_H+\left\langle\frac{\tilde{\lambda}}{2}\right\rangle_H+
\left\langle \frac{\tilde{\mu}^2}{\tilde{\mu}+\tfrac{\tilde{\lambda}}{2}}\right\rangle-\left\langle \frac{\tilde{\mu}}{\tilde{\mu}+\tfrac{\tilde{\lambda}}{2}}\right\rangle^2\left\langle\tilde{\mu}+\frac{\tilde{\lambda}}{2}\right\rangle_H-\langle \mu \rangle.
$$
By $\langle \cdot \rangle$ we have denoted the mean value of the function over the interval $[0,1]$, while by
$\langle \cdot \rangle_H$ we have denoted the harmonic mean of the function over the interval $[0,1]$.

In \cite{Horneuvel12} we were able to obtain explicit formula for layered material in the case $\gamma \in (0,\infty]$ only for very specific materials where $\lambda=0$.
\begin{remark}\label{rem:poveznica}
 Under the assumption that
$\tfrac{h}{\eh} \to  \gamma \in (0,\infty)$ the quadratic functional associated with the $\Gamma$-limit is given by
 $Q_{2,\gamma}:\M^{2}_{\sym}\to[0,\infty)$:
    \begin{eqnarray*}
      Q_{2,\gamma}(A)&:=&
      \inf_{B,\phi}\iint_{I\times Y}Q\left( x_3,
        y,\;\ \iota(x_3A+B)+(\nabla_y
        \phi\,,\,\tfrac{1}{\gamma}\partial_3\phi)\,\right)\,dy\,dx_3,
    \end{eqnarray*}
    where the infimum is taken over all $B\in\R^{2\times 2}_{\sym}$
    and $\phi\in \mathring{H}^1(I\times\mathcal Y,\R^3)$.
Using \cite[Lemma 5.2]{NeuVel-12} one can easily obtain that for
$A \in \M^{2}_{\sym}$ we have
$$ Q_2^{\text{rel}}(A)=\lim_{\gamma \to 0} Q_{2,\gamma}(A).$$
The continuity in $\gamma$, for all $\gamma \in [0,\infty]$, of the quadratic functional associated with the $\Gamma$-limit was already observed in the case of von-K\'arm\'an plate (see \cite{NeuVel-12}). The case of von-K\'{a}rm\'{a}n shell resembles the case of bending plate since we obtain that the continuity holds under the assumption that $\eh^2 \ll h \ll \eh$ as already commented in the introduction.
\end{remark}

\begin{proof}[Proof of Theorem~\ref{T:main} (lower bound)]
  Without loss of generality we may assume that
  $\fint_{\Omega}u^h\,dx=0$ and $\limsup_{h\to
    0}h^{-2}\mathcal E^{h,\eh}(u^h)<\infty$. In view of
  \eqref{ass:non-degenerate}, the sequence $u^h$ has finite
  bending energy and the sequence $E^h$ is bounded in
  $L^2(\Omega,\M^3)$ by using the inequality $|\sqrt{F^TF}-I|^2 \lesssim \dist^2(F,\SO 3)$, valid for an arbitrary $F \in \M^3$. Hence, from Theorem~\ref{T:FJM-compactness}
  we deduce that $u\in\WW(S,\R^3)$. By Lemma~\ref{kompts} (i) and Proposition~\ref{P:strain}
  (i) we can pass to a subsequence such that, for some $E\in L^2(\Omega\times Y; \R^{3\times 3})$,
  \begin{equation*}
    E^h\wtto E,
  \end{equation*}
  where $E$ can be written in the form of \eqref{P:strain:1a}. As explained in \cite{Neukamm-11}
(cf. \cite{FJM-02} for the corresponding argument in the homogeneous case), Taylor expansion of $W(\frac{x'}{\eh},\id+h E^h(x))$
combined with the lower semi-continuity of convex integral functionals
with respect to  weak two scale convergence (see e.g. \cite[Proposition~1.3]{Visintin-07}) yields the lower bound
\begin{align*}
&    \liminf\limits_{h\to 0} \frac{1}{h^2}\mathcal E^{h,\eh}(u^h) \geq
    \iint_{\Omega\times Y}Q( x_3,y,E(x,y))\,dy\,dx=
\\
& \iint_{\Omega\times Y}Q\Bigg( x_3,y,\iota(x_3\secf(x') + B(x'))+\mathcal{U}\big(\zeta(x,\cdot),\varphi(x,\cdot),g(x,\cdot,\cdot)\big) (x_3,y)\,\Bigg)\,dy\,dx,
\end{align*}
where we have used \eqref{P:strain:1a}.
Minimization over $B\in L^2(S,\M^2)$ and $\zeta \in L^2(S,\mathring{H}^1(\mathcal{Y},\R^2))$, $\varphi \in L^2(S,\mathring{H}^2(\mathcal{Y}))$, $g \in L^2(\Omega\times \mathcal{Y},\R^3)$ yields
  \begin{equation*}
    \liminf\limits_{h\to 0}\frac{1}{h^2}\mathcal E^{h,\eh}(u^h)\geq
    \int_{S}Q_{2}^{\text{rel}}(\secf(x'))\,dx'=\mathcal E(u).
  \end{equation*}
\end{proof}


\subsection{Upper bound}\label{Upper Bound}

It remains to prove the upper bound.
We modify the argumentation given in \cite{Horneuvel12} by adding additional oscillations.
To recover the matrix $B$ in the relaxation formula \ref{D:relaxation-formula} we use the same ansatz as in \cite{Horneuvel12}. The important observation, already present in \cite{Horneuvel12}, that we can not recover arbitrary matrix field $B$, but just the one necessary for relaxation. Namely on the flat parts of the isometry we don't need arbitrary matrix field to recover, but just zero field, since on the flat parts zero field solves the cell formula.
Since for $\Gamma$-limit it is enough to do the construction for dense subsets, first we will say which dense subset of $\WW(S,\R^3)$ is of interest.
The density of smooth isometries in $H^2$ isometric immersions is
established in \cite{Hornung-arma2} (cf. also \cite{Pakzad} for an earlier result in this direction).
The results in \cite{Hornung-arma2}  forces us to consider domains $S$
which are not only Lipschitz but also piecewise $C^1$. More precisely, we only need that the outer unit normal
be continuous away from a subset of $\partial S$ with length zero.

For a given $u\in \WW(S,\R^3)$ and for
a displacement $V\in H^2(S,\R^3)$ we denote by $q_V$ the quadratic form
$$
q_V = \sym \left((\nabla' u)^T(\nabla' V)\right).
$$
We denote by $\mathcal A(S)$ the set of all
$u\in \WW(S, \R^3)\cap C^{\infty}(\overline{S}, \R^3)$
with the property that
\begin{align*}
\Big\{ B\in C^{\infty}(\overline{S}, \M^2_{\sym}) &: B = 0 \mbox{ in a neighborhood of }\{x'\in S : \Pi(x') = 0 \}\Big\}
\\
&\subset \{q_V : V\in C^{\infty}(\overline{S}; \R^3)\}.
\end{align*}
In other words, if $u\in\mathcal A(S)$ and $B\in C^{\infty}(\overline{S}, \M^2_{\sym})$ is a matrix field
which vanishes in a neighborhood of $\{\Pi = 0\}$, then there exists a displacement $V\in C^{\infty}(\overline{S}; \R^3)$ such
that $q_V = B$.
The proof of the following lemma is given in \cite[Lemma 3.3]{Schmidt-07} while in \cite{Horneuvel12} it is written in this way, after appropriate substitution.

\begin{lemma}
\label{le71}
The set $\mathcal A(S)$ is dense in $\WW(S,\R^3)$ with respect to the strong $H^2$ topology.
\end{lemma}

Thanks to Lemma \ref{le71} it will be enough to construct recovery sequences for limiting
deformations belonging to $\mathcal{A}(S)$.

First  we present a construction assuming
additional information about the limit. Then we use the standard diagonalizing argument for the general case. The meaning of Lemma \ref{le71} is that we can recover the arbitrary matrix field out of the flat parts of the isometry.
\begin{remark}
For the homogenization of the bending shell one would need the stronger compactness result that would tell us what is exactly the (macroscopic) matrix field $B$ that can be used in the relaxation formula. This fact, together with lacking of results on density of smooth isometries in $H^2$ isometries for the general manifolds, are the main obstructions for deriving the homogenized bending shell model.
\end{remark}

\begin{lemma}
  \label{le73}
  Let $u\in\WW(S,\R^3)\cap W^{2,\infty}(S,\R^3)$ and let $V\in W^{2,\infty}(S, \R^3)$.
  Let $\zeta \in C^\infty_c(S,\mathring{C}^\infty(\mathcal{Y},\R^2))$, $\varphi \in C_c^\infty(S,\mathring{C}^\infty(\mathcal{Y}))$, $\bar{g} \in C_c^\infty(S,\mathring{C}^\infty(I \times \mathcal{Y},\R^3))$. Then there exists a sequence
    $(u^h)\subset H^1(\Omega, \R^3)$ such that
$u^h\to u$ and $\nabla_h u^h\to (\nabla'u,\,n)$ uniformly in $\Omega$ and
    \begin{multline}
      \label{le73-1a}
      \lim_{h \to 0} \frac{1}{h^2}\mathcal E^{h,\eh}(u^h)=\\
      \iint_{\Omega\times Y}Q\left(
         x_3,y,\;\iota(x_3\secf(x') + q_V(x'))+\mathcal{U}\big(\zeta(x,\cdot),\varphi(x,\cdot),g(x,\cdot,\cdot)\big) (x_3,y)\,\right)\,d  y \,d x_3 \,d x'.
    \end{multline}
\end{lemma}
\begin{proof}
First we start with the following Kirchhoff-Love type ansatz to which we add its linearization induced by the displacement $V$:
  \begin{eqnarray*}
    v^h(x)&:=&u(x')+hx_3n(x') + h\left(V(x') + h x_3\mu(x')\right),
  \end{eqnarray*}
where $\mu$ is given by
$$
\mu = (I - n\otimes n)(\d_1 V\wedge\d_2 u + \d_1 u\wedge\d_2 V).
$$
We set $R(x') = (\nabla'u(x'),\ n(x'))$. A straighforward computation
shows that
\begin{equation}
\label{vhers}
\nabla_hv^h = R + h\Big( (\nabla' V, \mu) + x_3(\nabla' n, 0) \Big) + h^2x_3(\nabla'\mu, 0).
\end{equation}
The actual recovery sequence $u^h$ is obtained by adding to $v^h$ the oscillating corrections
given by $\zeta$, $\varphi$, $\bar{g}$:
  \begin{eqnarray}&& \label{eq:15} \\
  \nonumber  u^h(x) &:=& v^h(x) -\eh^2 n(x') \varphi(x', \frac{x'}{\eh})\\  \nonumber & &+h\eh^2 x_3 R(x') \left(\begin{array}{c} \partial_{x_1} \varphi(x, \frac{x'}{\eh})+\tfrac{1}{\eh} \partial_{y_1} \varphi(x, \frac{x'}{\eh}) \\ \partial_{x_2} \varphi(x, \frac{x'}{\eh})+\tfrac{1}{\eh} \partial_{y_2} \varphi(x, \frac{x'}{\eh}) \\ 0 \end{array}\right)\\ \nonumber && +h\eh R(x')\left(\begin{array}{c} \zeta(x',\frac{x'}{\eh}) \\ 0 \end{array} \right)+h^2\int_{-1/2}^{x_3} R(x') \bar{g}(x',t,\frac{x'}{\eh}) \, d t.
  \end{eqnarray}
By the regularity of $V$, the uniform convergence of $u^h$ and $\nabla_hu^h$ is immediate.
Equation \eqref{vhers} implies
\begin{equation}
\label{fin-1}
\begin{split}
R^t\nabla_h u^h = \id & +\eh \left(\begin{array}{ccc} 0 & 0 &\partial_{y_1} \varphi \\  0 & 0 &\partial_{y_2} \varphi \\ -\partial_{y_1} \varphi & -\partial_{y_2} \varphi & 0\end{array} \right)
\\ &+ h\iota\left((\nabla' u)^t(\nabla' V) + x_3\secf\right) \\
&
+ h \left( (\mu\cdot\nabla' u)\otimes e_3 + e_3\otimes (n\cdot\nabla' V) \right)
\\
& + h \iota\left( x_3 \nabla_y^2 \varphi  +  \nabla_y \zeta \right) +h\left(\begin{array}{ccc} 0 & 0 & \bar{g}_1 \\ 0 &0 & \bar{g}_2 \\ 0 & 0 & \bar{g}_3  \end{array} \right)
\\
&-\eh^2\iota(\varphi\secf)  +\eh^2 \left(\begin{array}{ccc} 0 & 0 &\partial_{x_1} \varphi \\  0 & 0 &\partial_{x_2} \varphi \\ -\partial_{x_1} \varphi & -\partial_{x_2} \varphi & 0\end{array} \right) \\
&+h\eh^2 x_3   \iota \left(\begin{array}{cc} \partial_{x_1x_1} \varphi+\tfrac{1}{\eh}\partial_{x_1y_1} \varphi & \partial_{x_1x_2} \varphi+\tfrac{1}{\eh}\partial_{x_2y_1} \\ \partial_{x_1x_2} \varphi+\tfrac{1}{\eh}\partial_{x_1y_2} \varphi & \partial_{x_2x_2} \varphi+\tfrac{1}{\eh}\partial_{x_2y_2}   \end{array}\right)
\end{split}
\end{equation}

\begin{equation*}
\begin{split}
&+ h^2x_3R^t(\nabla'\mu, 0) + h\eh^2 x_3 (R^t\nabla' R) \left(\begin{array}{c} \partial_{x_1} \varphi+\tfrac{1}{\eh}\partial_{y_1} \varphi \\ \partial_{x_2} \varphi+\tfrac{1}{\eh}\partial_{y_2} \varphi \\ 0 \end{array}\right)\\ &  +h\eh \iota(\nabla_{x'} \zeta)+h\eh (R^t \nabla' R) \left(\begin{array}{c} \zeta \\ 0 \end{array} \right)+ \frac{h^2}{\eh} ( \int_{-1/2}^{x_3} \nabla_y \bar{g}\, , 0\,)\\ & +h^2( \int_{-1/2}^{x_3} \nabla_{x'} \bar{g}\, , 0\,)+h^2(R^t\nabla' R) \int_{-1/2}^{x_3} \bar{g};
\end{split}
\end{equation*}
the argument of the functions $\zeta,\varphi,g$  and their derivatives is $(x, x'/\eh)$.
From this we have, using
$
n\cdot\nabla' V + \mu\cdot\nabla' u = 0,
$
and the Assumption \ref{assumption:gamma}
\begin{equation}
\label{fin-10}
\begin{split}
(\nabla_h u^h)^t(\nabla_h u^h) = \id & + h\iota\left(2q_V + 2x_3\secf\right)
\\
& + h \iota\left( 2x_3 \nabla_y^2 \varphi  +  2\sym\nabla_y \zeta \right) +2h\left(\begin{array}{ccc} 0 & 0 & g_1 \\ 0 &0 & g_2 \\ g_1 & g_2 & g_3  \end{array} \right)+o(h),
\end{split}
\end{equation}
where $g=(\tfrac{1}{2}\bar{g}_1,\tfrac{1}{2} \bar{g}_2, \bar{g}_3)$ and $\tfrac{\|o(h)\|_{L^\infty}}{h} \to 0$.
Let us define
$$ E^h=\frac{\sqrt{(\nabla_h u^h)^t \nabla_h u^h}-I}{h}. $$
Using Taylor expansion
we deduce from \eqref{fin-10} that
\begin{align*}
E^h\stto E:=&
\iota\left(q_V + x_3\secf\right)+\mathcal{U}(\zeta(x',\cdot),\varphi(x',\cdot),g(x',\cdot))(x_3,y).
\end{align*}
%

Properties \eqref{ass:frame-indifference}, \eqref{ass:expansion} and
  \eqref{eq:94} yield
  \begin{equation}\label{dasistg2}
    \limsup\limits_{h\to 0}\left|\frac{1}{h^2}\mathcal
    E^{h,\eh}(u^h)-\int_{\Omega}Q(x_3,\tfrac{x'}{\eh},E^h(x))\,dx\right|=0.
  \end{equation}

  Hence, by \eqref{Qelliptisch} and by strong two scale convergence of $E^h$, we can pass to the
  limit in the second term in \eqref{dasistg2}.
\end{proof}

\begin{proof}[Proof of Theorem~\ref{T:main} (Upper bound)]
 The following is the standard diagonalization argument and is already given in \cite{Horneuvel12}.
 We may assume that $\mathcal E(u)<\infty$, so
  $u\in\WW(S,\R^3)$. Moreover, since $Q_{2}^{\text{rel}}$ is quadratic (cf. Lemma~\ref{L:Qgamma}), it suffices to prove the
  statement for a dense subset of $\WW(S,\R^3)$. Hence, by Lemma~\ref{le71}, we may assume without loss of generality that
  $u\in\mathcal A(S)$.
\\
By Lemma~\ref{L:Qgamma} there exist $B\in L^2(S,\M_{\sym}^{2})$ and
$\zeta \in L^2(S,\mathring H^1(\mathcal Y,\R^2))$, $\varphi \in L^2(S,\mathring{H}^2(\mathcal{Y}))$, $g \in L^2(S,L^2(I \times \mathcal{Y},\R^3))$ such that:
  \begin{equation}\label{eq:12}
    \mathcal E(u)=\iint_{\Omega\times Y}Q(x_3,y,\iota(x_3\secf + B)+\mathcal{U}(\zeta,\varphi,g))\,dydx.
  \end{equation}
  Since $B(x')$ depends linearly on $\secf(x')$, we know in addition that
  $B(x')=0$ for almost every $x'\in\{\,\secf=0\,\}$.

  By a density argument it suffices to show the following: There
  exists a doubly indexed sequence $u^{h,\delta}\in
  H^1(\Omega,\R^3)$ such that
  \begin{align}
    \label{eq:9}
    &\limsup\limits_{\delta\to 0}\limsup\limits_{h\to
      0}\|u^{h,\delta} - u\|_{H^1} = 0,\\
    \label{eq:11}
    &\limsup\limits_{h\to 0}\left|\frac{1}{h^2}\mathcal
      E^{h,\eh}(u^{h,\delta})-\mathcal E(u)\right|\lesssim\delta.
  \end{align}
If we establish this, then we can obtain the desired sequence by diagonalizing argument (e.~g. by appealing to
  \cite[Corollary 1.16]{Attouch-84}).

  We construct  $u^{h,\delta}$ as follows: By density, for each $\delta>0$
there exist maps $B^\delta\in C^\infty(\overline S,\R^{2\times 2}_{\sym})$ and
$\zeta^\delta \in C^\infty_c(S,\mathring{C}^\infty(\mathcal{Y},\R^2))$, $\varphi^\delta \in C_c^\infty(S,\mathring{C}^\infty(\mathcal{Y}))$, $g^\delta \in C_c^\infty(S,\mathring{C}^\infty(I \times \mathcal{Y},\R^3))$

 such that
  \begin{align}
    \label{eq:7}
    &\|B^\delta-B\|_{L^2(S)} + \|\mathcal{U}(\zeta^\delta,\varphi^\delta,g^\delta)-\mathcal{U}(\zeta,\varphi,g)\|_{L^2(\Omega\times
      Y)}\leq
    \delta^2,\\
    \label{eq:8}
    &B^\delta=0\text{ in a neighborhood of }\{\,\secf=0\,\}.
  \end{align}
\\
Since $u\in\mathcal A(S,\R^3)$ and due to \eqref{eq:8}, for each
$\delta > 0$ there exists a smooth displacement $V_{\delta}$ such that
$$
B^\delta = q_{V_{\delta}}.
$$
We apply Lemma~\ref{le73} to $u$ and $V_{\delta}$ to
obtain a sequence $u^{h,\delta}$ that converges uniformly to $u$ as $h\to 0$.
Hence \eqref{eq:9} is satisfied. Lemma~\ref{le73} also ensures that
  \begin{equation*}
    \lim\limits_{h\to 0}\frac{1}{h^2}\mathcal
    E^h(u^{h,\delta})=\iint_{\Omega\times Y}Q\left( x_3,y,\iota(x_3\secf + B^\delta)+\mathcal{U}(\zeta^\delta,\varphi^\delta,g^\delta)\right)\,dydx.
  \end{equation*}
  By continuity of the functional on the right-hand side,
  combined with \eqref{eq:7} and \eqref{eq:12}, the bound \eqref{eq:11}
  follows.
\end{proof}

\begin{remark} \label{slucaj2}
In the case $\eh^2 \sim h$ one can define the sequence, instead of the expression (\ref{eq:15}), which includes the oscillations which are of order different than $\eh$ and which possibly additionally relax the energy. This can be achieved e.g. by putting in (\ref{eq:15})  instead of $\varphi \in C_c^\infty(S,\mathring{C}^\infty(\mathcal{Y}))$ the function $\varphi_k \in C_c^\infty(S,\mathring{C}^\infty(k\mathcal{Y}))$. It can be easily seen that this is allowable recovery sequence which could have less energy than the original one. Namely, in the expression for the strain then would also appear the matrix $A^T A$, where
$$A= \left(\begin{array}{ccc} 0 & 0 &\partial_{y_1} \varphi_k \\  0 & 0 &\partial_{y_2} \varphi_k \\ -\partial_{y_1} \varphi_k & -\partial_{y_2} \varphi_k & 0\end{array} \right).$$
This nonlinear term would cause nonconvexity of the energy in $\nabla \varphi$ which has the consequence that the oscillations which are not of the order $\eh$ could also be energetically convenient.
This partially explains the lack of compactness, which is commented in Remark \ref{slucaj1}.
\end{remark}
\section{Appendix}
Let us recall some well-known properties of two scale convergence.
We
refer to \cite{Allaire-92, Visintin-06, MilTim07} for proofs
in the standard two scale setting and to \cite{Neukamm-10} for the easy
adaption to the notion of two scale convergence considered here.
\begin{lemma}\label{kompts}
  \begin{enumerate}[(i)]
  \item Any sequence that is bounded in $L^2(\tilde{\Omega})$ admits a two scale
    convergent subsequence.
  \item Let $\t f\in L^2(\tilde{\Omega}\times Y)$ and let $(f^h)_{h>0}\subset L^2(\tilde{\Omega})$ be such that $f^h\wtto \t{f}$.
   Then $f^h\wto\int_Y\t f(\cdot,y)\,dy$ weakly in $L^2(\tilde{\Omega})$.
  \item Let $f^0 \in L^2(\tilde{\Omega})$ and $(f^h)_{h>0} \subset L^2(\tilde{\Omega})$ be such that $f^h\wto
    f^0$ weakly in
    $L^2(\tilde{\Omega})$. Then (after passing to subsequences) we have $f^h\wtto f^0(x)+\t f$ for some $\t f\in L^2(\tilde{\Omega}\times
    Y)$ with $\int_Y\t f(\cdot,y)\,dy=0$ almost everywhere in $\tilde{\Omega}$. $\t f$ is uniquely characterized by the fact that $\int_{\tilde{\Omega}} f^h \psi(x,\tfrac{x}{\eh}) \ud x \to \int_{\tilde{\Omega}} \int_{\mathcal{Y}} \t f(x,y) \psi(x,y) \ud y\ud x$, for every $\psi  \in C_0^\infty
(\tilde{\Omega},\mathring{C}^\infty(\mathcal{Y}))$.
  \item Let $f^0 \in H^1(\tilde{\Omega})$ and $(f^h)_{h>0} \subset H^1(\tilde{\Omega})$ be such that $f^h\to f^0$ strongly in
    $L^2(\tilde{\Omega})$. Then $f^h\stto f^0$, where we extend $f^0$ trivially to $\tilde{\Omega}\times Y$.
  \item Let $f^0$ and $f^h\in H^1(\tilde{S})$ be such that $f^h\wto f^0$ weakly in $H^1(\tilde{S})$.
    Then (after passing to subsequences)
    \begin{equation*}
      \nabla' f^h\wtto \nabla 'f^0+\nabla_y\phi,
    \end{equation*}
    for some $\phi\in L^2(\tilde{S}, H^1(\mathcal Y))$.
  \end{enumerate}
\end{lemma}
The following lemma gives us the characterization of two scale limits of scaled gradients when $h \ll \varepsilon(h)$.
For the proof of it see
\cite[Theorem~6.3.3]{Neukamm-10}.

\begin{lemma}\label{L:two scale}
  Let $u^0 \in H^1(\tilde{\Omega},\R^3)$ and $(u^h)_{h>0}\subset H^1(\tilde{\Omega},\R^3)$ be such that $u^h\wto u^0$ weakly in
  $H^1(\tilde{\Omega},\R^3)$ and $\liminf_{h\to 0}\|\nabla_h u^h\|_{L^2}<\infty$. Under the assumption $\lim_{h\to 0} \tfrac{h}{\eh}=0$
 there exist
$\phi \in L^2(\tilde{S},\mathring H^1(\mathcal Y,\R^3))$, $d \in L^2(\tilde{S},L^2(I\times \mathcal{Y},\R^3))$ such that
(after passing to subsequences)
\begin{equation}
\nabla_h u^h \wtto (\nabla'u^0\,,\,0)+(\nabla_y\phi\,,\,d).
\end{equation}
Here $u^0$ is the weak limit of $u^h$ i.e. $\int_I u^h$ in $H^1(\tilde{\Omega},\R^3)$ i.e. $H^1(\tilde{S},\R^3)$.
\end{lemma}

The following two lemmas are proved in \cite{Velcic-12}. The first one tells us that two scale limit, like weak limit, behaves nicely when we multiply it with a bounded sequence which converges in measure. The second one tells us that if we have a sequence of vector fields that is closer to a sequence of gradients than the order of oscillations, then the form of the oscillatory part of its two scale limit is the second gradient, like the two scale limit of gradients. This is also valid under subtle assumption that the sequence of gradients itself converges only strongly in $L^2$ (but not necessarily weakly in $H^1$).
\begin{lemma} \label{mnozenjets}
Let $(u^h)_{h>0}$ be a bounded sequence in $L^2(\tilde{\Omega})$ which
two scale converges to $u_0(x,y) \in L^2(\tilde{\Omega} \times \mathcal{Y})$. Let
$(v^h)_{h>0}$ be a sequence bounded in $L^\infty(\tilde{\Omega})$ which
converges in measure to $v_0 \in L^{\infty}(\tilde{\Omega})$. Then $v^h
u^h \wtto v_0(x) u_0(x,y)$.
\end{lemma}
\begin{lemma} \label{pomoooc}
Let $(u^h)_{h>0}$ be a sequence which converges strongly
to $u$ in $H^1(\tilde{\Omega})$ and $(v^h)_{h>0}$ be a sequence which
is bounded in $H^1(\tilde{\Omega},\R^n)$ such that for each $h>0$
\begin{equation} \label{oggrada}
\| \nabla u^h-v^h\|_{L^2} \leq C\eta(h),
\end{equation}
for some $C>0$ and $\eta(h)$ which satisfies $\lim_{h \to 0} \tfrac{\eta(h)}{\eh}=0$. Then for any subsequence of $(\nabla v^h)_{h>0}$ which converges two scale there exists a unique $v \in
L^2(\Omega,\mathring{H}^2(\mathcal Y))$ such that $\nabla v^h
\wtto \nabla^2 u(x)+\nabla^2_y
v(x,y)$.
\end{lemma}

{\bf Acknowledgements.} The work on this paper was mainly done while the author was affiliated with Peter Hornung's group in Max Planck Institute for Mathematics in the Sciences in Leipzig, Germany. The author was supported by Deutsche Forschungsgemeinschaft grant no.
HO-4697/1-1.

\bibliographystyle{plain}

\end{document}